\documentclass[11pt]{article}
\pdfoutput=1

\usepackage{graphicx}
\usepackage{amsmath}
\usepackage{amssymb}
\usepackage{amsthm}
\usepackage[margin=1in]{geometry}
\usepackage{hyperref}
\usepackage{tikz}

\newcommand{\Hclos}{\widehat{\mathbb{H}}}

\newtheorem{theorem}{Theorem}[section]
\newtheorem{lemma}[theorem]{Lemma}
\newtheorem{proposition}[theorem]{Proposition}
\newtheorem{corollary}[theorem]{Corollary}

\theoremstyle{definition}
\newtheorem{definition}[theorem]{Definition}
\newtheorem{remark}[theorem]{Remark}

\title{Analogue of Cardy's formula on the 60-degree parallelogram: a modular approach}
\author{Cody R. Strouse}
\date{\today}

\begin{document}
\maketitle

\begin{abstract}
The \emph{Cardy--Smirnov function} of a four-marked planar domain encodes the conjectural scaling limit of the crossing probability for critical percolation. Cardy~\cite{Cardy1992} predicted a closed formula for this limit on the rectangle in the cases of both site and bond percolation, and Smirnov~\cite{Smirnov} proved the formula together with conformal invariance for critical site percolation on the triangular lattice. Kleban and Zagier~\cite{KZ} later showed that on the rectangle the function admits a modular interpretation: it is determined by a modular functional equation together with a mild analytic ansatz. We carry out the analogue on the $\pi/3$ parallelogram. We derive a closed conformal-map formula for the Cardy--Smirnov function as an incomplete beta integral, prove a closed modular formula realizing it as an integral of the modular form $\eta(\tau)^2\eta(3\tau)^2$, and establish a uniqueness theorem showing that a single functional equation, together with a $q$-expansion ansatz and a nondegeneracy condition, determines the function uniquely.
\end{abstract}

\section{Introduction}

\subsection{Percolation and the Cardy--Smirnov function}

\emph{Site percolation} is a random model on a lattice: each lattice site (vertex) is independently colored \emph{open} (with probability $p$) or \emph{closed} (with probability $1-p$). Fix a simply connected planar domain $R$ with four marked boundary points dividing $\partial R$ into four arcs $A,B,C,D$ in clockwise cyclic order, and consider the model at the critical probability $p = p_c$, the threshold above which an infinite connected cluster of open sites appears almost surely. (For site percolation on the triangular lattice, $p_c=1/2$~\cite{Kesten1982}.) For a lattice of mesh $a$ contained in $R$, let $\mathcal C_a(R)$ denote the event that there exists a connected cluster of open sites joining the arc $A$ to the arc $C$, and write $\mathbb P_a(\mathcal C_a(R))$ for its probability. The \emph{Cardy--Smirnov function of the domain $R$} is the limit
\[
h(R) \;:=\; \lim_{a\to 0}\,\mathbb P_a(\mathcal C_a(R)),
\]
when this limit exists.

Cardy~\cite{Cardy1992} predicted, and Smirnov~\cite{Smirnov} proved in the special case of critical site percolation on the triangular lattice, that $h$ is a \emph{conformal invariant}: it depends on $R$ only through its conformal equivalence class. Thus $h$ may be regarded as a function on the one-parameter family of conformal equivalence classes of four-marked simply connected domains (parametrized, for instance, by the cross-ratio of the four boundary points under any conformal map to the upper half-plane), and computing $h$ on any one representative determines its value on all conformally equivalent ones. This is the point of view we adopt throughout. Cardy originally conjectured his formula for both site and bond percolation, and the same conformal invariance is expected to hold on other lattices in the same generality; it is widely believed that the scaling limit exists in that generality, but this remains open. We refer the reader to Grimmett's exposition~\cite{Grimmett} for the precise statements of Smirnov's results and for further context.

Cardy's original formula was stated for the rectangle and is recalled in~\S\ref{subsec:cardy-original} below. Smirnov's reformulation on the equilateral triangle, due originally to L.\ Carleson (unpublished; see~\cite{Smirnov}), takes a strikingly simple form: with vertices $V_1, V_2, V_3$ in counterclockwise cyclic order and a fourth marked point $X$ on the side $V_2 V_3$, the Cardy--Smirnov function for the crossing probability from the side $V_1 V_2$ to the boundary segment $X V_3$ equals the Euclidean length $|XV_3|$ when the triangle is scaled to have unit side. We refer to this as the \emph{Carleson form} of Cardy's formula. The two representatives---rectangle and equilateral triangle---are conformally equivalent as four-marked domains, and Cardy's rectangle formula and Smirnov--Carleson's triangle formula are equivalent descriptions of the same conformal invariant.

A subsequent contribution due to Kleban and Zagier~\cite{KZ} was a reformulation of the rectangle case in terms of modular forms: they expressed the Cardy--Smirnov function of the rectangle as an integral of a weight-$2$ modular form, and characterized it uniquely via a modular functional equation. This is the result we recall in~\S\ref{subsec:KZ} and that we extend in the present paper.

The purpose of this paper is to carry out the analogous modular analysis on a natural next family of domains: the parallelogram with smallest interior angle $\pi/3$. Starting from a Schwarz--Christoffel map construction, we prove three results about the Cardy--Smirnov function on this domain: a closed conformal-map formula as an incomplete beta integral, a closed modular formula realizing it as an integral of the modular form $\eta(\tau)^2\eta(3\tau)^2$, and a uniqueness theorem; see Theorems~\ref{thm:cardy-parallelogram}, \ref{thm:main-almost-modular}, and~\ref{thm:uniqueness} below.

See Figures~\ref{fig:scaling-limit} and~\ref{fig:rect-vs-parallelogram} for the scaling limit on the parallelogram and the comparison between the rectangle and parallelogram settings.

\begin{figure}[htbp]
\centering
\includegraphics[width=\textwidth]{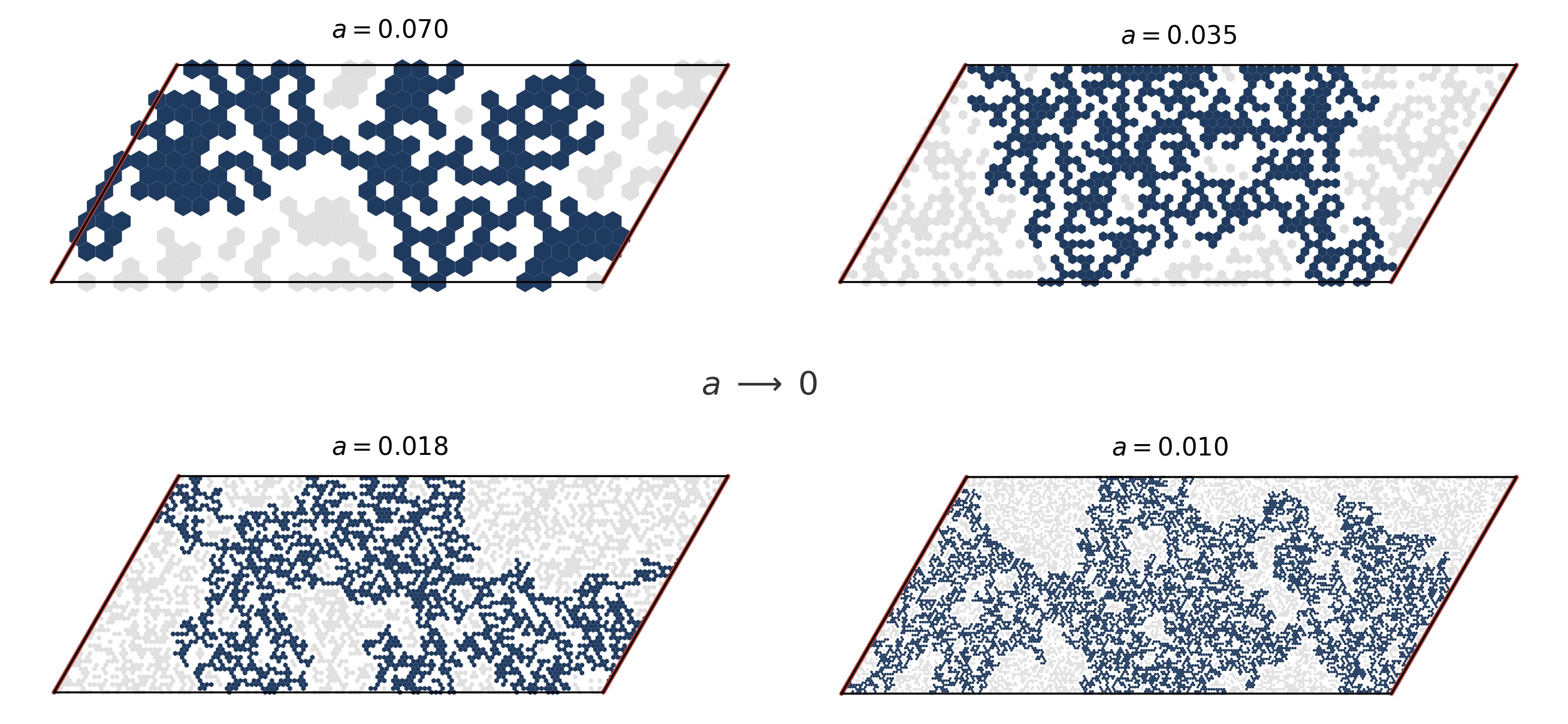}
\caption{Critical site percolation on the $\pi/3$ parallelogram at successively finer lattice spacings $a$. The highlighted cluster is the largest connected cluster of open sites joining the left side to the right side (a \emph{left--right crossing}). Conjecturally, as $a\to 0$ such clusters converge in distribution to a random fractal whose crossing probability is a conformal invariant of the domain.}
\label{fig:scaling-limit}
\end{figure}
\begin{figure}[htbp]
\centering
\includegraphics[width=\textwidth]{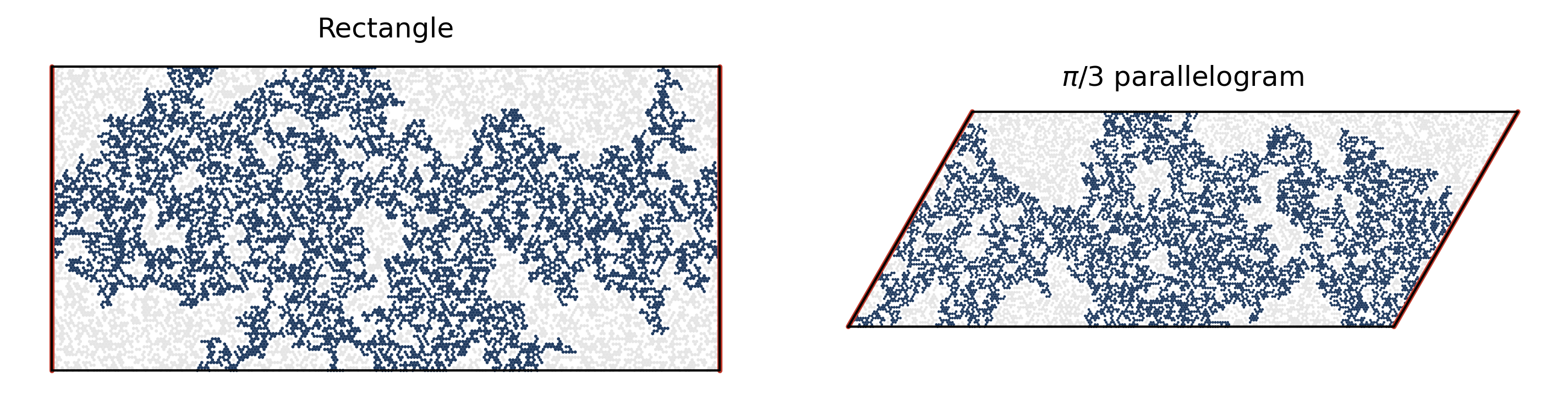}
\caption{Critical site percolation on the triangular lattice ($p=p_c=1/2$) in a rectangle (the setting of Kleban--Zagier~\cite{KZ}) and in a $\pi/3$ parallelogram, the setting of this work.}
\label{fig:rect-vs-parallelogram}
\end{figure}

Throughout, $\mathbb H = \{\tau \in \mathbb C : \Im\tau > 0\}$ denotes the open upper half-plane, and $\Hclos = \mathbb H \cup \mathbb R \cup \{\infty\}$ its closure in the Riemann sphere $\mathbb P^1(\mathbb C)$. We set $q = e^{2\pi i\tau}$ and denote by $\eta(\tau) = q^{1/24}\prod_{n\ge 1}(1-q^n)$ the Dedekind eta function. The Euler beta function is $B(a,b) = \Gamma(a)\Gamma(b)/\Gamma(a+b)$, and ${}_2F_1(a,b;c;z) = \sum_{n\ge 0}\frac{(a)_n(b)_n}{(c)_n\,n!}\,z^n$ is the Gauss hypergeometric function, with $(x)_n = x(x+1)\cdots(x+n-1)$ the Pochhammer symbol.

\subsection{Cardy's formula}
\label{subsec:cardy-original}

Conformally mapping the rectangle to the upper half-plane sends its four corners to four real points $0$, $\lambda$, $1$, $\infty$ (after a M\"obius normalization), and $\lambda\in(0,1)$ is then the cross-ratio of these four points---the standard conformal invariant of a four-marked rectangle. For a rectangle of aspect ratio $r$ the relevant value is $\lambda = \lambda(ir)$, where
\[
\lambda(\tau) \;=\; 16\,\frac{\eta(\tau/2)^8\,\eta(2\tau)^{16}}{\eta(\tau)^{24}}
\]
is the classical modular lambda function. Cardy's original 1992 prediction~\cite{Cardy1992} for the Cardy--Smirnov function $\Pi_h^{\mathrm{rect}}$ of the rectangle is
\begin{equation}
\label{eq:cardy-original}
\Pi_h^{\mathrm{rect}}(r) \;=\; \frac{2\pi\sqrt{3}}{\Gamma(1/3)^3}\,\lambda^{1/3}\,{}_2F_1\!\left(\tfrac13,\tfrac23;\tfrac43;\lambda\right),\qquad \lambda = \lambda(ir).
\end{equation}
Smirnov~\cite{Smirnov} subsequently proved the Carleson form of Cardy's formula on the equilateral triangle and, in the same paper, proved conformal invariance of percolation crossing probabilities in this setting. Together, these two results imply~\eqref{eq:cardy-original} on the rectangle via the conformal bijection between the rectangle and the equilateral triangle as four-marked domains. The combined effect of~\cite{Cardy1992,Smirnov} is that, for critical site percolation on the triangular lattice, the Cardy--Smirnov function~\eqref{eq:cardy-original} is established as a theorem and has the explicit closed form predicted by Cardy.

\subsection{The work of Kleban--Zagier}
\label{subsec:KZ}

Let $R_r$ denote the $1\times r$ rectangle, regarded as a four-marked domain with the corners as distinguished points, and consider the Cardy--Smirnov function $\Pi_h^{\mathrm{rect}}(r)$, as given in~\eqref{eq:cardy-original}. Kleban and Zagier give a \emph{modular interpretation} of $\Pi_h^{\mathrm{rect}}$: a closed formula expressing it as an integral of a weight-$2$ modular form, together with a uniqueness theorem characterizing it via a single functional equation.

\begin{theorem}[Kleban--Zagier~\cite{KZ}]
\label{thm:KZ}
The following hold.
\begin{enumerate}
\item\emph{(Formula.)} The Cardy--Smirnov function of the rectangle is given by
\[
\Pi_h^{\mathrm{rect}}(r)=\frac{2^{7/3}\pi^2}{\sqrt{3}\,\Gamma(1/3)^3}\int_{r}^{\infty}\eta(it)^4\,dt.
\]
\item\emph{(Characterization.)} Let $F:(0,\infty)\to\mathbb R$ be any function such that
\begin{itemize}
\item[\textup{(i)}] $F(r)=\sum_{n\ge 0} a_n\,\widehat q^{\,2n+\alpha}$ with $\widehat q=e^{-\pi r}$, some $\alpha>0$, and $a_0\ne 0$ (an \emph{even conformal block of dimension $\alpha$} in the terminology of~\cite{KZ}: only even powers of $\widehat q$ accompany $\widehat q^{\,\alpha}$);
\item[\textup{(ii)}] $F(1/r)=1-F(r)$.
\end{itemize}
Then $\alpha=1/3$ and $F=\Pi_h^{\mathrm{rect}}$.
\end{enumerate}
\end{theorem}

The curious feature of Theorem~\ref{thm:KZ}(2) is that a single modular-style symmetry, together with the assumption that $F$ is built from a $q$-series block, is enough to determine the Cardy--Smirnov function uniquely. The evenness in~(i) cannot be dropped: Kleban and Zagier show (their Theorem~2) that without it the same functional equation admits a one-parameter family of solutions, one for each $\alpha\in(0,\tfrac12]$. The condition is best understood through cusp widths. The cross-ratio $\lambda$ is a hauptmodul for $\Gamma(2)$, whose cusp at $i\infty$ has width $2$, so the nome native to the rectangle problem is $\widehat q=e^{\pi i\tau}$, and a general block in $\widehat q$ yields a scalar transformation law only under $\tau\mapsto\tau+2$; evenness is exactly the requirement that the block descend to the width-one nome $q=\widehat q^{\,2}$, upgrading the transformation law to $\tau\mapsto\tau+1$ and hence the ambient group from the theta group $\langle S,T^2\rangle=\Gamma_\theta$ (two cusps) to $\langle S,T\rangle=\mathrm{PSL}_2(\mathbb Z)$ (one cusp). In the parallelogram setting the hauptmodul $v$ of Theorem~\ref{thm:cardy-parallelogram} lives on $\Gamma_0(3)$, whose cusp at $i\infty$ already has width $1$; hypothesis~(i) of Theorem~\ref{thm:uniqueness} is therefore stated directly in the width-one nome $q=e^{2\pi i\tau}$, and the analogue of the parity condition is vacuous.

\subsection{Main results}
\label{subsec:main-results}

Let $P_r$ denote the closed $\pi/3$ parallelogram with vertices $0$, $r$, $r+e^{i\pi/3}$, $e^{i\pi/3}$ in cyclic order, with boundary $\partial P_r$. We define the Cardy--Smirnov function of $P_r$, which we denote by $\Pi_h^{\mathrm{par}}(r)$, purely conformally: it is the value of the conformal invariant $h$ at the conformal equivalence class of $P_r$. Concretely, by the Riemann mapping theorem there is a conformal map $P_r \to \Delta$ onto the equilateral triangle $\Delta$, unique once we require it to take the four marked vertices of $P_r$ to four prescribed marked boundary points of $\Delta$ (the Riemann mapping theorem gives existence and uniqueness up to a M\"obius automorphism of the disk, a three-real-parameter ambiguity that is removed by fixing the images of the four marked points). With this normalization, $\Pi_h^{\mathrm{par}}(r)$ is the corresponding Carleson length, as in~\S\ref{subsec:cardy-original}. Probabilistically, this is conjecturally the limiting crossing probability between the two short sides for critical percolation on any reasonable lattice, and is known to be such for critical site percolation on the triangular lattice~\cite{Smirnov}.

We prove three main results. The first is an analogue of Cardy's original formula~\eqref{eq:cardy-original}: a closed expression for $\Pi_h^{\mathrm{par}}(r)$ as an incomplete beta integral.

\begin{theorem}[Analogue of Cardy's formula for the $\pi/3$-parallelogram]
\label{thm:cardy-parallelogram}
For $r>0$ let $u(r)\in(0,1)$ denote the prevertex parameter of the Schwarz--Christoffel uniformization $\Hclos\to P_r$ that sends $(0,u,1,\infty)$ to the four vertices of $P_r$ in cyclic order, and set $v(r):=1-u(r)$; these are well defined by Theorem~\ref{thm:r-ratio} and Lemma~\ref{lem:endpoint-r}. Then $u,v:(0,\infty)\to(0,1)$ are real-analytic, and the following hold.
\begin{enumerate}
\item[\textup{(i)}] The Cardy--Smirnov function of $P_r$ is the incomplete beta integral
\[
\Pi_h^{\mathrm{par}}(r)\;=\;\frac{1}{B(\tfrac13,\tfrac13)}\int_{u(r)}^1 w^{-2/3}(1-w)^{-2/3}\,dw \;=\; \frac{3\,v(r)^{1/3}}{B(\tfrac13,\tfrac13)}\,{}_2F_1\!\left(\tfrac13,\tfrac23;\tfrac43;v(r)\right).
\]
Here $B(a,b) = \Gamma(a)\Gamma(b)/\Gamma(a+b)$ is the Euler beta function; in particular
\[
B(\tfrac13,\tfrac13) \;=\; \frac{\Gamma(1/3)^2}{\Gamma(2/3)} \;=\; \frac{\sqrt{3}\,\Gamma(1/3)^3}{2\pi}.
\]
\item[\textup{(ii)}] Writing $\tau=ir/\sqrt3$ and $\eta$ for the Dedekind eta function, the prevertex parameter is the eta quotient
\begin{equation}
\label{eq:v-eta}
v(r) \;=\; \frac{27\,\eta(3\tau)^{12}}{\eta(\tau)^{12} + 27\,\eta(3\tau)^{12}},\qquad u(r) \;=\; 1 - v(r).
\end{equation}
\end{enumerate}
\end{theorem}

Part~(i) is proved in~\S\ref{sec:SC} by an explicit conformal map. Part~(ii) is the modular content of the theorem: it asserts that the geometric prevertex parameter, a priori defined only through a Schwarz--Christoffel integral, is a modular function of $\tau=ir/\sqrt3$ for the congruence group $\Gamma_0(3)$. It is proved in~\S\ref{sec:monodromy} (Lemma~\ref{lem:u-v-explicit}).

The second main result is an analogue of part~(1) of Theorem~\ref{thm:KZ}: a closed formula realizing $\Pi_h^{\mathrm{par}}$ as an integral of a modular form.

\begin{theorem}[Modular formula for the crossing probability]
\label{thm:main-almost-modular}
$\Pi_h^{\mathrm{par}}(r)$ is given by
\[
\Pi_h^{\mathrm{par}}(r)\;=\;\frac{4\pi^2}{\Gamma(1/3)^3}\int_r^{\infty}\eta\!\left(\tfrac{i}{\sqrt 3}\,t\right)^{\!2}\eta\!\left(i\sqrt 3\,t\right)^{\!2}\,dt.
\]
\end{theorem}

The third result is an analogue of Theorem~\ref{thm:KZ}(2): a uniqueness statement showing that the Fricke symmetry, together with a $q$-block ansatz and a nondegeneracy condition, forces the Cardy--Smirnov function. Here and throughout, the \emph{cusp} $i\infty$ is the boundary point of $\mathbb H$ approached as $\Im\tau \to \infty$; a holomorphic function on $\mathbb H$ invariant under $\tau \mapsto \tau + 1$ has a $q$-expansion there, and we speak of its behavior ``at the cusp''.

\begin{theorem}[Uniqueness]
\label{thm:uniqueness}
Let $F:\mathbb H\to\mathbb C$ be holomorphic and suppose that
\begin{enumerate}
\item[\textup{(i)}] $F$ admits a convergent $q$-expansion at the cusp $i\infty$ of the form
\[
F(\tau)=\sum_{n=0}^{\infty} a_n\,q^{n+\alpha},\qquad a_0\ne 0,\ \alpha>0;
\]
\item[\textup{(ii)}] $F$ satisfies the functional equation $F\!\left(-\tfrac{1}{3\tau}\right)=1-F(\tau)$ for all $\tau\in\mathbb H$;
\item[\textup{(iii)}] the derivative $f := dF/d\tau$ is nowhere vanishing on $\mathbb H$.
\end{enumerate}
Then $\alpha=1/3$ and $F(\tau) = \Pi_h^{\mathrm{par}}(-i\sqrt 3 \,\tau)$, the holomorphic extension of the Cardy--Smirnov function to $\mathbb H$ constructed in Theorem~\ref{thm:modular-derivative}\textup{(i)}.
\end{theorem}

Hypothesis~(iii) is a nondegeneracy condition: geometrically it says that $F$ is locally injective, i.e.\ it extends a conformal map. It holds automatically for the Cardy--Smirnov function, whose $\tau$-derivative is the nowhere-vanishing weight-$2$ form $c_\eta\,\eta(\tau)^2\eta(3\tau)^2$ (Theorem~\ref{thm:modular-derivative}). The condition is genuinely necessary here, in contrast to the rectangular setting of Kleban--Zagier: we explain in Remark~\ref{rem:ghost} below why dropping it admits a second, spurious solution.

\subsection{Strategy and structure of the paper}

Our proof of Theorem~\ref{thm:main-almost-modular} proceeds in three layers.

\emph{Conformal layer.} Section~\ref{sec:SC} uses the Schwarz--Christoffel map from the upper half-plane to $P_r$, composed with the normalized map to the equilateral triangle, to express $\Pi_h^{\mathrm{par}}(r)$ as an incomplete beta integral in a single prevertex parameter $u\in(0,1)$. Combined with the Carleson form of Cardy's formula on the triangle, this gives an explicit hypergeometric expression for $\Pi_h^{\mathrm{par}}(r)$.

\emph{Hypergeometric layer.} Section~\ref{sec:monodromy} shows that the horizontal side length $r=r(u)$ is a ratio of two solutions of a Gauss hypergeometric equation, and computes the projective monodromy of that equation explicitly in the coordinate $\tau=ir/\sqrt3$: it is generated by $\tau\mapsto\tau+1$ and $\tau\mapsto\tau/(1-3\tau)$, and equals $\Gamma_0(3)$. The prevertex $u$ is thereby exhibited as a hauptmodul for $\Gamma_0(3)$, and the symmetry $u\mapsto1-u$ of the hypergeometric equation is exhibited as the Fricke involution $\tau\mapsto-1/(3\tau)$.

\emph{Modular layer.} Section~\ref{sec:proof-main} uses the cubic theta functions of Borwein and Borwein~\cite{Borwein} to identify the $\tau$-derivative of $\Pi_h^{\mathrm{par}}$ as the weight-$2$ cusp form $\eta(\tau)^2\eta(3\tau)^2$, and integrates to obtain Theorem~\ref{thm:main-almost-modular}. Section~\ref{sec:rigidity} proves Theorem~\ref{thm:uniqueness} by a four-step modular argument on $\Gamma_0^+(3)$. Section~\ref{sec:discussion} discusses obstructions to extending the argument to other angles $\pi/n$. Appendix~\ref{app:SC} recalls the Schwarz--Christoffel derivative formula, and Appendix~\ref{app:endpoint} proves that $r(u)$ is a real-analytic increasing bijection $(0,1)\to(0,\infty)$.

\section{The Schwarz--Christoffel construction and a conformal formula for \texorpdfstring{$\Pi_h^{\mathrm{par}}$}{Pi\_h\^{}par}}
\label{sec:SC}

\subsection{Strategy}
Let $\Delta$ denote the closed equilateral triangle with vertices $0$, $1$, $e^{i\pi/3}$. We construct a conformal map $g : P_r \to \Delta$ from the parallelogram onto the triangle, sending the four marked vertices of $P_r$ to four marked boundary points of $\Delta$, and then apply the Carleson form of Cardy's formula on $\Delta$ to read off $\Pi_h^{\mathrm{par}}(r)$ as a Euclidean length. The map $g$ is built by composing two Schwarz--Christoffel maps through the upper half-plane:
\[
g \;=\; \Phi_\Delta \circ f_u^{-1} \;:\; P_r \;\xrightarrow{\;f_u^{-1}\;}\;\mathbb H\;\xrightarrow{\;\Phi_\Delta\;}\;\Delta,
\]
where $f_u : \mathbb H \to P_r$ is the Schwarz--Christoffel uniformization of the parallelogram and $\Phi_\Delta : \mathbb H \to \Delta$ is the analogous uniformization of the triangle. The composite is illustrated in Figure~\ref{fig:crossing_map}; explicit formulas are derived in~\S\ref{subsec:two-SC-maps} and~\S\ref{subsec:hypergeometric}.

\begin{figure}[htbp]
    \centering
    \resizebox{0.98\textwidth}{!}{%
    \begin{tikzpicture}[>=latex, scale=1.0]
    \path[use as bounding box] (-0.5,-1.6) rectangle (12.5,3.0);
    \tikzstyle{crossing} = [orange, thick, smooth, tension=0.7]

    \begin{scope}[shift={(0,0)}]
        \coordinate (O) at (0,0);
        \coordinate (Lc) at (2.5,0);
        \coordinate (Topr) at (1,1.732);
        \coordinate (TopR) at (3.5,1.732);
        \fill[gray!5] (O)--(Lc)--(TopR)--(Topr)--cycle;
        \draw[red, very thick] (O)--(Lc);
        \draw[blue, very thick] (O)--(Topr);
        \draw[green!60!black, very thick] (Lc)--(TopR);
        \draw[black, very thick] (Topr)--(TopR);
        \draw[crossing] (0.5,0.866) to[out=20,in=160] (1.5,1.2) to[out=-20,in=180] (3.0,1.0);
        \node[below left] at (O) {$0$};
        \node[below] at (Lc) {$r$};
        \node at (1.75,2.2) {\textbf{Parallelogram } $P_r$};
    \end{scope}

    \begin{scope}[shift={(5.0,0)}]
        \fill[gray!5] (-0.4,0) rectangle (3.4,1.9);
        \draw[blue, very thick] (-0.4,0)--(0,0);
        \draw[red, very thick] (0,0)--(1.5,0);
        \draw[green!60!black, very thick] (1.5,0)--(2.5,0);
        \draw[black, very thick] (2.5,0)--(3.4,0);
        \draw[crossing] (-0.25,0.25) to[out=70,in=180] (1.0,1.1) to[out=0,in=140] (2.05,0.15);
        \node[below] at (0,0) {$0$};
        \node[below] at (1.5,0) {$u$};
        \node[below] at (2.5,0) {$1$};
        \node[below right] at (3.4,0) {$\infty$};
        \node at (1.5,2.2) {\textbf{Upper half-plane } $\Hclos$};
    \end{scope}

    \begin{scope}[shift={(10,0)}]
        \coordinate (t0) at (0,0);
        \coordinate (tg) at (1.5,0);
        \coordinate (t1) at (2.5,0);
        \coordinate (top) at (1.25,2.165);
        \fill[gray!5] (t0)--(t1)--(top)--cycle;
        \draw[blue, very thick] (t0)--(top);
        \draw[red, very thick] (t0)--(tg);
        \draw[green!60!black, very thick] (tg)--(t1);
        \draw[black, very thick] (t1)--(top);
        \draw[crossing] (0.625,1.08) .. controls (1.0,1.0) and (1.2,0.5) .. (1.8,0);
        \node[below left] at (t0) {$0$};
        \node[below] at (tg) {$g(r)$};
        \node[below right] at (t1) {$1$};
        \node at (1.25,2.5) {\textbf{Triangle } $\Delta$};
    \end{scope}

    \draw[->, thick] (3.65,1.08) -- node[midway, above] {$f_u^{-1}$} (4.5,1.08);
    \draw[->, thick] (8.7,1.08) -- node[midway, above] {$\Phi_\Delta$} (9.9,1.08);

    \draw[->, thick, dashed] (1.75,-0.6) to[bend right=12] node[midway, below] {$g \;=\; \Phi_\Delta \circ f_u^{-1}$} (11.25,-0.6);
    \end{tikzpicture}
    }
    \caption{The composite conformal map $g = \Phi_\Delta \circ f_u^{-1} : P_r \to \Hclos \to \Delta$. The four boundary sides are color-coded consistently across the three domains. A left--right crossing in $P_r$ (orange) corresponds to a path in $\Delta$ joining the side $[0,e^{i\pi/3}]$ to the segment $[g(r), 1]$; by Theorem~\ref{thm:cardy-carleson}, $\Pi_h^{\mathrm{par}}(r) = 1 - g(r)$.}
    \label{fig:crossing_map}
\end{figure}
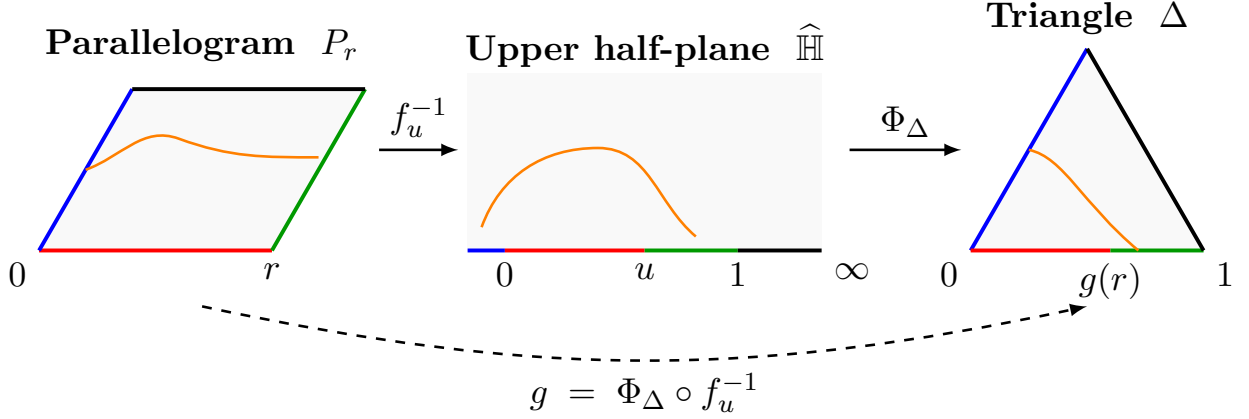

\subsection{The triangular domain}

We use $\Delta$ to denote the equilateral triangle with vertices $0$, $1$, $e^{i\pi/3}$, and recall the Carleson form of Cardy's formula, which will be the input to all our crossing calculations.

\begin{theorem}[Smirnov~\cite{Smirnov}; Carleson form]
\label{thm:cardy-carleson}
Let $\Delta ABC$ be an equilateral triangle of side length $1$ and let $X$ be a point on the side $BC$. The Cardy--Smirnov function for the crossing probability from the side $AB$ to the boundary segment $XC\subset BC$ equals the Euclidean length of $XC$. (That is, in critical site percolation on the triangular lattice, the probability of an open crossing from $AB$ to $XC$ converges as the mesh tends to zero to the length $|XC|$.) The triangular form of the statement, reformulating Cardy's rectangle formula~\eqref{eq:cardy-original} as a Euclidean length, is due to L.\ Carleson (unpublished; see~\cite{Smirnov}).
\end{theorem}

We use the principal branch of the complex power $z \mapsto z^{\alpha}$ on $\mathbb C \setminus (-\infty, 0]$ throughout, with $z^\alpha := \exp(\alpha \operatorname{Log} z)$ for $\operatorname{Log}$ the principal logarithm; in particular, for $u \in (0,1)$ the quantities $u^{-2/3}$ and $(1-u)^{-2/3}$ are real and positive. For factors such as $(w-u)^{-1/3}$ and $(w-1)^{-2/3}$ that have branch cuts crossing the real interval of integration, we evaluate them at real $w$ as boundary values from the upper half-plane: $(w-u)^{-1/3} := \lim_{\epsilon \downarrow 0} (w + i\epsilon - u)^{-1/3}$, and similarly for $(w-1)^{-2/3}$.

\begin{definition}[Triangle uniformization]
\label{def:PhiDelta}
Define $\Phi_\Delta:\Hclos\to\mathbb C$ by
\[
\Phi_\Delta(z):=\frac{\int_0^z w^{-2/3}(1-w)^{-2/3}\,dw}{\int_0^1 w^{-2/3}(1-w)^{-2/3}\,dw}.
\]
\end{definition}

By the classical theory of Schwarz--Christoffel maps (see Appendix~\ref{app:SC} for the derivation), $\Phi_\Delta$ extends to a conformal bijection $\mathbb H \to \Delta$ sending the prevertices $0, 1, \infty$ to the vertices $0, 1, e^{i\pi/3}$ of $\Delta$ respectively, with $\Phi_\Delta(0) = 0$ and $\Phi_\Delta(1) = 1$ by construction.

\subsection{The parallelogram domain}
\label{subsec:two-SC-maps}

The $\pi/3$ parallelogram $P_r$ has interior angles $(\pi/3, 2\pi/3, \pi/3, 2\pi/3)$ at its four vertices, taken in cyclic order. By the Schwarz--Christoffel theory (Appendix~\ref{app:SC}), a conformal bijection $\Hclos \to P_r$ sending four prevertices $(0, u, 1, \infty) \in \mathbb R \cup \{\infty\}$ to the four vertices of $P_r$ in cyclic order has the form
\begin{equation}
\label{eq:fu-def}
f_u(z) \;=\; C_{\mathrm{SC}}(u) \int_0^z w^{-2/3}(w-u)^{-1/3}(w-1)^{-2/3}\,dw \qquad (z \in \Hclos),
\end{equation}
for a normalizing constant $C_{\mathrm{SC}}(u) \in \mathbb C^\times$.

The parameter $u \in (0,1)$ is not free: the four-marked parallelogram $P_r$ is determined (up to conformal equivalence) by its conformal modulus, which translates into a unique value of $u$ for each side length $r$. We make this precise. For each $u \in (0,1)$, formula~\eqref{eq:fu-def} produces \emph{some} $\pi/3$ parallelogram; varying $u$ over $(0,1)$ traces out (after rescaling) the one-parameter family of conformal equivalence classes of $\pi/3$ parallelograms, parametrized by their horizontal side length. We will show in Theorem~\ref{thm:r-ratio} and Lemma~\ref{lem:endpoint-r} that the resulting side length is a strictly increasing function $r = r(u)$, whose inverse $u = u(r) \in (0,1)$ is the unique prevertex parameter producing the parallelogram with side length $r$. Throughout the rest of this section, we work with arbitrary $u \in (0,1)$ and write $r = r(u)$; at the end we invert.

\begin{definition}[Maps used in the construction]
\label{def:map-dictionary}
For $u \in (0,1)$, let $f_u : \Hclos \to P_{r(u)}$ be the Schwarz--Christoffel map~\eqref{eq:fu-def} sending the prevertices $(0, u, 1, \infty)$ to the vertices $(0, r(u), r(u) + e^{i\pi/3}, e^{i\pi/3})$ respectively. Write $\phi_1 := f_u^{-1} : P_{r(u)} \to \Hclos$, and define the composite
\[
g := \Phi_\Delta \circ \phi_1 \;:\; P_{r(u)} \longrightarrow \Delta.
\]
\end{definition}

The logical content of the construction is summarized by the diagram
\[
\begin{array}{ccccc}
P_{r(u)} & \xrightarrow{\;\phi_1\;} & \Hclos & \xrightarrow{\;\Phi_\Delta\;} & \Delta. \\[4pt]
\multicolumn{5}{c}{\xrightarrow[\qquad g\;=\;\Phi_\Delta\,\circ\,\phi_1\qquad]{}}
\end{array}
\]

\subsection{The side length $r$ as a ratio of period integrals}
\label{subsec:hypergeometric}

The side length $r$ and the crossing probability $\Pi_h^{\mathrm{par}}(r)$ are both expressed through the following two integrals, which we call \emph{period integrals} because they are integrals of a multi-valued algebraic differential form over paths between branch points --- precisely the type of integral that arises as a period of an algebraic curve.

\begin{definition}[Period integrals]
\label{def:AB}
For $u\in(0,1)$, set
\begin{align*}
N(u) &:= \int_0^u w^{-2/3}(u-w)^{-1/3}(1-w)^{-2/3}\,dw, \\
D(u) &:= \int_u^1 w^{-2/3}(w-u)^{-1/3}(1-w)^{-2/3}\,dw.
\end{align*}
\end{definition}
\begin{theorem}[Hypergeometric formula for $r(u)$]
\label{thm:r-ratio}
For each $u \in (0,1)$, the parallelogram $P_{r(u)}$ produced by~\eqref{eq:fu-def} has horizontal side length
\[
r(u)\;=\;\frac{N(u)}{D(u)}\;=\;\frac{{}_2F_1(\tfrac13,\tfrac23;1;u)}{{}_2F_1(\tfrac13,\tfrac23;1;1-u)}.
\]
That is, $r$ is a ratio of two solutions of the Gauss hypergeometric equation
\begin{equation}
\label{eq:hgeo-ODE}
u(1-u)\,F''(u)+(1-2u)\,F'(u)-\tfrac29\,F(u)\;=\;0
\end{equation}
defined on $(0,1)$.
\end{theorem}

\begin{proof}
For $w \in (0, u)$, taking boundary values from $\mathbb H$ as in~\S\ref{subsec:two-SC-maps}, the principal-branch boundary values satisfy
\[
\lim_{\epsilon\downarrow 0}(w + i\epsilon - u)^{-1/3} \;=\; e^{-\pi i/3}\,(u-w)^{-1/3}, \qquad \lim_{\epsilon\downarrow 0}(w + i\epsilon - 1)^{-2/3} \;=\; e^{-2\pi i/3}\,(1-w)^{-2/3},
\]
while for $w \in (u, 1)$ only the factor $(w-1)^{-2/3}$ contributes a phase $e^{-2\pi i/3}$. Therefore
\[
f_u(u) - f_u(0) \;=\; C_{\mathrm{SC}}(u)\,e^{-\pi i/3}\,e^{-2\pi i/3}\,N(u) \;=\; -C_{\mathrm{SC}}(u)\,N(u),
\]
and
\[
f_u(1) - f_u(u) \;=\; C_{\mathrm{SC}}(u)\,e^{-2\pi i/3}\,D(u).
\]
The conditions $f_u(u) = r$ and $f_u(1) - f_u(u) = e^{i\pi/3}$ then give $C_{\mathrm{SC}}(u) = -r/N(u) = -1/D(u)$, and hence $r(u) = N(u)/D(u)$.

To identify $N$ and $D$ with hypergeometric functions, substitute $w = u t$ in $N$:
\[
N(u) \;=\; u^{1 - 2/3 - 1/3}\,\int_0^1 t^{-2/3}(1-t)^{-1/3}(1 - u t)^{-2/3}\,dt \;=\; \int_0^1 t^{-2/3}(1-t)^{-1/3}(1 - u t)^{-2/3}\,dt.
\]
By Euler's integral representation $_2F_1(a, b; c; z) = \tfrac{\Gamma(c)}{\Gamma(b)\Gamma(c-b)}\int_0^1 t^{b-1}(1-t)^{c-b-1}(1-zt)^{-a}\,dt$ (valid for $\Re c > \Re b > 0$ and $|z| < 1$, with appropriate branch conventions at $z = 1$), the integrand corresponds to parameters $b = 1/3$, $c - b = 2/3$, $a = 2/3$, so $c = 1$, giving
\[
N(u) \;=\; B(\tfrac13, \tfrac23)\,{}_2F_1(\tfrac23, \tfrac13; 1; u) \;=\; B(\tfrac13, \tfrac23)\,{}_2F_1(\tfrac13, \tfrac23; 1; u).
\]
Substituting $w = u + (1-u)t$ in $D$ and a parallel computation give $D(u) = B(\tfrac13, \tfrac23)\,{}_2F_1(\tfrac13, \tfrac23; 1; 1 - u)$. The $B$-factor cancels in the ratio, yielding the claimed hypergeometric expression. Both $N$ and $D$ satisfy~\eqref{eq:hgeo-ODE}: $N(u)$ is holomorphic at $u = 0$ (the Euler integral converges there), while $D(u)$ is holomorphic at $u = 1$. They are linearly independent---their ratio is nonconstant by Lemma~\ref{lem:endpoint-r} below---and hence span the solution space of~\eqref{eq:hgeo-ODE} on $(0,1)$.
\end{proof}

The side length is a genuine coordinate on the family of $\pi/3$ parallelograms:

\begin{lemma}[Boundary behavior of the prevertex parameter]
\label{lem:endpoint-r}
The function $u\mapsto r(u)$ of Theorem~\ref{thm:r-ratio} is a real-analytic, strictly increasing bijection $(0,1)\to(0,\infty)$; in particular $r(u)\to 0$ as $u\to0^+$ and $r(u)\to+\infty$ as $u\to1^-$.
\end{lemma}

The proof, elementary but slightly technical, is deferred to Appendix~\ref{app:endpoint}.

\begin{proof}[Proof of Theorem~\ref{thm:cardy-parallelogram}\textup{(i)}]
Part~(ii) is proved in~\S\ref{sec:monodromy} (Lemma~\ref{lem:u-v-explicit}); we prove part~(i) here.

Set $v(r) := 1 - u(r)$, where $u(r)$ is the inverse of the function $r = r(u)$ established in Theorem~\ref{thm:r-ratio} (this inverse exists by the monotonicity of Lemma~\ref{lem:endpoint-r}, proved in Appendix~\ref{app:endpoint}). By construction, $g(r) = \Phi_\Delta(u(r))$, and Theorem~\ref{thm:cardy-carleson} applied to the conformal image $g(P_r) = \Delta$ gives $\Pi_h^{\mathrm{par}}(r) = 1 - g(r)$. Since $\Phi_\Delta(1) = 1$, the normalization in Definition~\ref{def:PhiDelta} supplies the denominator, so
\[
\Pi_h^{\mathrm{par}}(r) \;=\; 1 - g(r) \;=\; \frac{1}{B(\tfrac13,\tfrac13)}\int_{u(r)}^1 w^{-2/3}(1-w)^{-2/3}\,dw.
\]
The standard incomplete-beta-to-${}_2F_1$ identity
\begin{equation}
\label{eq:incomplete-beta-2F1}
\int_0^z w^{-2/3}(1-w)^{-2/3}\,dw \;=\; 3\,z^{1/3}\,{}_2F_1(\tfrac13, \tfrac23; \tfrac43; z) \qquad (z \in [0,1]),
\end{equation}
which follows by expanding $(1-w)^{-2/3}$ in a binomial series and integrating term by term, using $\tfrac{1}{k+1/3}=3\,\tfrac{(1/3)_k}{(4/3)_k}$, gives the second equality after the substitution $w = 1 - \tilde w$ converts the tail $\int_{u(r)}^1$ on the left into $\int_0^{v(r)}$ with $v = 1 - u$:
\[
\int_{u(r)}^1 w^{-2/3}(1-w)^{-2/3}\,dw \;=\; \int_0^{v(r)} \tilde w^{-2/3}(1-\tilde w)^{-2/3}\,d\tilde w \;=\; 3\,v(r)^{1/3}\,{}_2F_1(\tfrac13, \tfrac23; \tfrac43; v(r)).
\]
This yields the second form, completing the proof of part~(i).
\end{proof}

\section{Monodromy and modular uniformization}
\label{sec:monodromy}

Our goal in this section is to introduce a uniformizing coordinate $\tau\in\mathbb H$ in which the prevertex parameter $u$ of \S\ref{sec:SC} becomes a modular function for $\Gamma_0(3)$, and to identify that function explicitly as an eta quotient.

\subsection{Notation}
\label{subsec:group-notation}

For an integer $N\ge 1$, the \emph{Hecke congruence subgroup of level $N$} is
\[
\Gamma_0(N)=\left\{\begin{pmatrix}a&b\\ c&d\end{pmatrix}\in \mathrm{SL}_2(\mathbb Z):c\equiv 0\pmod N\right\},
\]
acting on $\mathbb H$ by M\"obius transformations; we write $\overline{\Gamma}$ for the image of a subgroup $\Gamma\subseteq\mathrm{SL}_2(\mathbb Z)$ in $\mathrm{PSL}_2(\mathbb R)=\mathrm{SL}_2(\mathbb R)/\{\pm I\}$. Throughout we abbreviate
\[
T=\begin{pmatrix}1&1\\0&1\end{pmatrix},\qquad
L=\begin{pmatrix}1&0\\3&1\end{pmatrix},\qquad
W_3=\begin{pmatrix}0&-1\\3&0\end{pmatrix},
\]
so that $T\tau=\tau+1$, $L\tau=\tau/(3\tau+1)$, and $W_3\tau=-1/(3\tau)$. Both $T$ and $L$ lie in $\Gamma_0(3)$; the \emph{Fricke involution} $W_3$ does not, but normalizes $\Gamma_0(3)$.

We recall some standard vocabulary, used throughout. A \emph{Fuchsian group} is a discrete subgroup of $\mathrm{PSL}_2(\mathbb R)$. A nontrivial element of $\mathrm{PSL}_2(\mathbb R)$ is \emph{elliptic} if it fixes a point of $\mathbb H$; equivalently, $|\operatorname{tr}\gamma|<2$ for either matrix representative. A point of $\mathbb H$ fixed by some nontrivial element of a Fuchsian group $\Gamma$ is an \emph{elliptic point} of $\Gamma$, and its \emph{order} is the order of its stabilizer in $\Gamma$, a finite cyclic group. For $\Gamma$ commensurable with $\mathrm{PSL}_2(\mathbb Z)$, the \emph{cusps} of $\Gamma$ are the $\Gamma$-orbits of $\mathbb Q\cup\{\infty\}$, and the \emph{width} of the cusp $\infty$ is the least $h>0$ such that $\tau\mapsto\tau+h$ lies in $\Gamma$.

At a regular singular point of a second-order linear equation, the indicial equation has two roots, called the \emph{local exponents}: solutions behave like the corresponding powers of the local coordinate, with a logarithm appearing when the exponents coincide. The \emph{exponent difference} at the point is the difference of the two roots. For the hypergeometric equation with parameters $(a,b;c)$ the exponent pairs at $0$, $1$, $\infty$ are $(0,1-c)$, $(0,c-a-b)$ and $(a,b)$, respectively~\cite[pp.~17--18]{BeukersNotes}. For the Schwarz map we refer to~\cite[pp.~24--31]{BeukersNotes} and~\cite{Yoshida}; for modular-curve background, to~\cite[Chs.~2--3]{Diamond}.

\subsection{The Schwarz map and the coordinate \texorpdfstring{$\tau$}{tau}}
\label{subsec:uniformization}

Given a second-order linear ODE on $\mathbb P^1$ with three regular singular points, placed at $0,1,\infty$ after a M\"obius change of variable, the \emph{Schwarz map} is the ratio $u\mapsto f_1(u)/f_2(u)$ of two linearly independent local solutions. It is defined on any simply connected domain avoiding the singular points, and it is well defined there up to postcomposition with a M\"obius transformation, because two bases of the two-dimensional solution space differ by an element of $\mathrm{GL}_2(\mathbb C)$. Analytic continuation along a loop in $\mathbb P^1\setminus\{0,1,\infty\}$ replaces the Schwarz map by its postcomposition with a M\"obius transformation; the group of all transformations so obtained is the \emph{projective monodromy group} of the equation.

By Theorem~\ref{thm:r-ratio}, the side length $r(u)=N(u)/D(u)$ is itself a Schwarz map for the hypergeometric equation~\eqref{eq:hgeo-ODE}. We now fix a particular normalization of this map by definition, and compute its monodromy explicitly in Proposition~\ref{prop:monodromy-group}; Remark~\ref{rem:normalization} explains in what sense the normalization is forced.

\begin{definition}[The uniformizing coordinate]
\label{def:tau}
For $u$ in the upper half-plane $\mathbb H_u:=\{\Im u>0\}$, let
\[
\tau(u)\;:=\;\frac{i}{\sqrt3}\,\frac{N(u)}{D(u)},
\]
where $N,D$ are the period integrals of Definition~\ref{def:AB}, continued analytically from $(0,1)$ into $\mathbb H_u$. For $u\in(0,1)$ this reads $\tau=ir(u)/\sqrt3$, so that $\tau$ takes values on the positive imaginary axis.
\end{definition}

The key analytic input for this section is the following pair of connection formulae, describing $N$ and $D$ near the singular points $u=1$ and $u=0$ respectively. Here and below,
\[
G(z)\;:=\;{}_2F_1(\tfrac13,\tfrac23;1;z),\qquad B\;:=\;B(\tfrac13,\tfrac23)=\Gamma(\tfrac13)\Gamma(\tfrac23)=\frac{\pi}{\sin(\pi/3)}=\frac{2\pi}{\sqrt3},
\]
so that $N(u)=B\,G(u)$ and $D(u)=B\,G(1-u)$ by the proof of Theorem~\ref{thm:r-ratio}.

\begin{lemma}[Connection formulae]
\label{lem:connection}
There are functions $A$ and $\widetilde A$, holomorphic in a neighborhood of $u=1$ and of $u=0$ respectively, such that
\[
N(u)\;=\;-\frac{\sqrt3}{2\pi}\,D(u)\,\log(1-u)\;+\;A(u),
\qquad
D(u)\;=\;-\frac{\sqrt3}{2\pi}\,N(u)\,\log u\;+\;\widetilde A(u).
\]
Moreover $A(1)=\dfrac{3\sqrt3\,\log 3}{2\pi}\,B$.
\end{lemma}

\begin{proof}
For $c=a+b$ the Gauss function has the logarithmic expansion~\cite[\S15.8(ii)]{DLMF} (stated there for Olver's normalized function $\mathbf F=F/\Gamma(c)$)
\[
{}_2F_1(a,b;a+b;z)=\frac{\Gamma(a+b)}{\Gamma(a)\Gamma(b)}\sum_{n\ge0}\frac{(a)_n(b)_n}{(n!)^2}\Bigl[2\psi(n+1)-\psi(a+n)-\psi(b+n)-\log(1-z)\Bigr](1-z)^n,
\]
where $\psi=\Gamma'/\Gamma$. With $(a,b)=(\tfrac13,\tfrac23)$ the prefactor is $1/\bigl(\Gamma(\tfrac13)\Gamma(\tfrac23)\bigr)=\sqrt3/(2\pi)$, and the coefficient of $-\log(1-z)$ is $\tfrac{\sqrt3}{2\pi}\sum_n\frac{(a)_n(b)_n}{(n!)^2}(1-z)^n=\tfrac{\sqrt3}{2\pi}G(1-z)$. Putting $z=u$ and multiplying by $B$ gives the first formula, with
\[
A(1)=B\cdot\frac{\sqrt3}{2\pi}\Bigl(2\psi(1)-\psi(\tfrac13)-\psi(\tfrac23)\Bigr).
\]
By the reflection and duplication identities for $\psi$ one has $\psi(\tfrac13)+\psi(\tfrac23)=-2\gamma-3\log3$, while $\psi(1)=-\gamma$; hence the bracket equals $3\log 3$, giving the stated value of $A(1)$. The second formula is the first with $u$ replaced by $1-u$.
\end{proof}

\subsection{The Schwarz triangle}

Write
\[
\rho\;:=\;\frac{-3+i\sqrt3}{6}\;=\;-\frac12+\frac{i}{2\sqrt3},
\]
and let
\begin{equation}
\label{eq:schwarz-triangle}
\mathcal T\;:=\;\Bigl\{\tau\in\mathbb H:\ -\tfrac12<\Re\tau<0,\ \bigl|\tau+\tfrac13\bigr|>\tfrac13\Bigr\}.
\end{equation}
This is a hyperbolic triangle with vertices $0$, $i\infty$ and $\rho$: its sides are the geodesics $\Re\tau=0$ (joining $0$ to $i\infty$), $\Re\tau=-\tfrac12$ (joining $\rho$ to $i\infty$), and the semicircle $|\tau+\tfrac13|=\tfrac13$ (joining $0$ to $\rho$). The semicircle is tangent to the imaginary axis at $0$, so the angles at the two ideal vertices $0$ and $i\infty$ vanish, while a direct computation gives interior angle $\pi/3$ at $\rho$.

\begin{proposition}[Schwarz]
\label{prop:schwarz-triangle}
The map $\tau$ of Definition~\ref{def:tau} extends continuously to $\overline{\mathbb H_u}\cup\{\infty\}$ and is a conformal bijection of $\mathbb H_u$ onto $\mathcal T$, carrying the boundary intervals and the singular points as follows:
\[
(0,1)\ \mapsto\ \{\Re\tau=0\},\qquad
(1,\infty)\ \mapsto\ \{\Re\tau=-\tfrac12\},\qquad
(-\infty,0)\ \mapsto\ \{|\tau+\tfrac13|=\tfrac13\},
\]
and $\tau(0)=0$, $\tau(1)=i\infty$, $\tau(\infty)=\rho$.
\end{proposition}

\begin{proof}
The local exponent pairs of~\eqref{eq:hgeo-ODE} at $u=0,1,\infty$ are $(0,1-c)=(0,0)$, $(0,c-a-b)=(0,0)$ and $(a,b)=(\tfrac13,\tfrac23)$, so the exponent differences are $\lambda_0=\lambda_1=0$ and $\lambda_\infty=\tfrac13$~\cite[p.~17]{BeukersNotes}. By Schwarz's theorem on the Schwarz map (see~\cite[Thm.~2.2.12]{BeukersNotes} or~\cite{Yoshida}), the image of $\mathbb H_u$ under any branch is a curvilinear triangle with angles $\pi\lambda_0=0$, $\pi\lambda_1=0$ and $\pi\lambda_\infty=\pi/3$, and the map is a conformal bijection onto it that extends continuously to the closure, sending the three boundary intervals to the three sides.

It remains to identify the particular triangle for the branch of Definition~\ref{def:tau}. Since $N$ and $D$ are real and positive on $(0,1)$, we have $\tau\bigl((0,1)\bigr)\subset i\mathbb R_{>0}$, and Lemma~\ref{lem:endpoint-r} gives $\tau(0)=0$ and $\tau(1)=i\infty$; thus the side joining the two ideal vertices is the imaginary axis. To locate the third vertex, take $u\to1$ inside $\mathbb H_u$ and write $1-u=\varepsilon e^{i(\theta-\pi)}$ with $\varepsilon\downarrow0$ and $\theta\in(0,\pi)$. By Lemma~\ref{lem:connection} and $D(1)=B$,
\[
\frac{N(u)}{D(u)}\;=\;-\frac{\sqrt3}{2\pi}\log(1-u)+\frac{A(1)}{B}+o(1)
\;=\;\frac{\sqrt3}{2\pi}\Bigl(\log\frac{1}{\varepsilon}+3\log 3\Bigr)-\frac{i\sqrt3}{2\pi}(\theta-\pi)+o(1),
\]
so that
\begin{equation}
\label{eq:cusp-asymptotics}
\Re\tau\;=\;\frac{\theta-\pi}{2\pi}+o(1),\qquad
\Im\tau\;=\;\frac{1}{2\pi}\log\frac{27}{\varepsilon}+o(1).
\end{equation}
As $\theta$ runs over $(0,\pi)$ the real part fills $(-\tfrac12,0)$, so near the vertex $i\infty$ the triangle is the strip $-\tfrac12<\Re\tau<0$, with $\theta\to\pi$ (i.e.\ $u\uparrow1$ along $(0,1)$) giving the side $\Re\tau=0$ and $\theta\to0$ (i.e.\ $u\downarrow1$ along $(1,\infty)$) giving the side $\Re\tau=-\tfrac12$. The two sides of $\mathcal T$ emanating from $i\infty$ are therefore the vertical lines $\Re\tau=0$ and $\Re\tau=-\tfrac12$, and the third vertex $\tau(\infty)$ lies on the latter.

It remains to identify the third side, the image of $(-\infty,0)$. Being the image of a segment of $\mathbb R$ under a Schwarz map, it is an arc of a circle or line; it passes through $\tau(0)=0$, and the angle of the triangle there is $0$, so it is tangent to the imaginary axis at $0$. A circle through $0$ tangent to the imaginary axis has its centre on $\mathbb R$; write it as $|\tau+R|=R$ with $R>0$ (the sign because the triangle lies in $\Re\tau<0$). Provided $R>\tfrac14$, this circle meets the line $\Re\tau=-\tfrac12$ at $P=-\tfrac12+i\sqrt{R-\tfrac14}$. Writing $P=-R+R(\cos\alpha,\sin\alpha)$ we have $\cos\alpha=(R-\tfrac12)/R$, and the unit tangent to the circle at $P$ pointing back towards $0$ is $(\sin\alpha,-\cos\alpha)$; the interior angle $\theta$ of the triangle at $P$, measured between this tangent and the upward vertical $(0,1)$, therefore satisfies
\[
\cos\theta\;=\;-\cos\alpha\;=\;\frac{\tfrac12-R}{R}.
\]
Setting $\theta=\pi/3$ gives $R=\tfrac13$, hence $P=-\tfrac12+\tfrac{i}{2\sqrt3}=\rho$, and this value of $R$ is the unique one. Hence the image is $\mathcal T$, and $\tau(\infty)=\rho$.
\end{proof}

\begin{figure}[htbp]
\centering
\begin{tikzpicture}[scale=5.4, >=latex]
\begin{scope}
\clip (-0.5,0) rectangle (0,1.02);
\fill[gray!22] (-0.5,0) rectangle (0,1.02);
\fill[white] (-0.333333,0) circle (0.333333);
\end{scope}
\begin{scope}
\clip (0,0) rectangle (0.5,1.02);
\fill[gray!8] (0,0) rectangle (0.5,1.02);
\fill[white] (0.333333,0) circle (0.333333);
\end{scope}
\draw[gray!55] (-0.72,0) -- (0.72,0);
\foreach \x/\lab in {-0.5/{-\frac12},0/{0},0.5/{\frac12}}
  \draw[gray!55] (\x,0.014) -- (\x,-0.014) node[below=0pt,black,scale=0.78] {$\lab$};
\draw[thick] (0,0) arc (0:120:0.333333);
\draw[thick] (0,0) arc (180:60:0.333333);
\draw[thick] (-0.5,0.288675) -- (-0.5,1.02);
\draw[thick] (0.5,0.288675) -- (0.5,1.02);
\draw[semithick,densely dashed] (0,0) -- (0,1.02);
\fill (-0.5,0.288675) circle (0.0065) node[left=2pt,scale=0.85] {$\rho$};
\fill (0.5,0.288675) circle (0.0065) node[right=2pt,scale=0.85] {$\rho+1$};
\node[scale=0.85] at (0,1.09) {$i\infty$};
\node[scale=0.95] at (-0.27,0.72) {$\mathcal T$};
\node[scale=0.95] at (0.27,0.72) {$\mathcal T^{*}$};
\draw[<->,gray!70] (-0.5,0.86) -- node[above,black,scale=0.75,inner sep=1.5pt]{$T$} (0.5,0.86);
\draw[<->,gray!70] (-0.166667,0.288675) to[bend right=38] node[below,black,scale=0.75,inner sep=1.5pt]{$L$} (0.166667,0.288675);
\end{tikzpicture}
\caption{The Schwarz triangle $\mathcal T$ of~\eqref{eq:schwarz-triangle} (darker) is the image of the upper half $u$-plane under the coordinate $\tau$ of Definition~\ref{def:tau}. Its vertices $0$, $i\infty$, $\rho=(-3+i\sqrt3)/6$ are the images of the singular points $u=0,1,\infty$ of~\eqref{eq:hgeo-ODE}, and its angles there are $0$, $0$, $\pi/3$, matching the exponent differences $\lambda_0=\lambda_1=0$, $\lambda_\infty=\tfrac13$. Reflection in the dashed side $\Re\tau=0$ produces $\mathcal T^{*}$, the image of the lower half $u$-plane. Their union $\mathcal F=\mathcal T\cup i\mathbb R_{>0}\cup\mathcal T^{*}$ is the standard fundamental domain for $\Gamma_0(3)$; the side pairings are $T:\tau\mapsto\tau+1$ on the two vertical sides and $L:\tau\mapsto\tau/(3\tau+1)$ on the two arcs, and these are exactly the generators produced by the monodromy computation of Proposition~\ref{prop:monodromy-group}.}
\label{fig:hyperbolic-schwarz}
\end{figure}
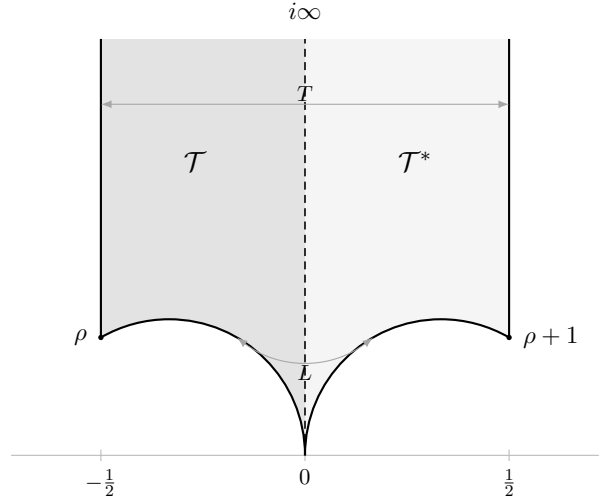

\subsection{Projective monodromy}

\begin{proposition}[Explicit projective monodromy]
\label{prop:monodromy-group}
Let $\ell_0$ and $\ell_1$ be positively oriented simple loops in $\mathbb C\setminus\{0,1\}$ encircling $u=0$ and $u=1$ respectively. Analytic continuation of the coordinate $\tau$ of Definition~\ref{def:tau} along $\ell_0$ and along $\ell_1$ acts on $\tau$ by
\[
\ell_0:\ \tau\longmapsto \frac{\tau}{1-3\tau}\quad(\text{the matrix }L^{-1}),\qquad
\ell_1:\ \tau\longmapsto \tau+1\quad(\text{the matrix }T).
\]
Consequently the projective monodromy group of~\eqref{eq:hgeo-ODE} in the coordinate $\tau$ is
\[
\overline M\;=\;\bigl\langle\,T,\ L\,\bigr\rangle\;\subseteq\;\overline{\Gamma_0(3)}.
\]
Moreover $L^{-1}T=\left(\begin{smallmatrix}1&1\\-3&-2\end{smallmatrix}\right)$ has trace $-1$, hence order $3$ in $\mathrm{PSL}_2(\mathbb R)$, and fixes $\rho=\tau(\infty)$; it is the local monodromy at $u=\infty$ for a suitable choice of base path.
\end{proposition}

\begin{proof}
Near $u=1$ the function $D$ is holomorphic and $D(1)=B\ne0$, so by Lemma~\ref{lem:connection}
\[
\tau=\frac{i}{\sqrt3}\cdot\frac{N}{D}=-\frac{i}{2\pi}\log(1-u)+h(u),\qquad h:=\frac{i}{\sqrt3}\cdot\frac{A}{D}\ \text{holomorphic at }u=1.
\]
Continuation along $\ell_1$ sends $\log(1-u)\mapsto\log(1-u)+2\pi i$ and fixes $h$, hence $\tau\mapsto\tau+1$.

Near $u=0$ the function $N$ is holomorphic with $N(0)=B\ne0$. Since $\tau=\tfrac{i}{\sqrt3}N/D$ we have $1/\tau=-i\sqrt3\,D/N$, and substituting the second connection formula of Lemma~\ref{lem:connection} gives
\[
\frac{1}{\tau}\;=\;-i\sqrt3\left(-\frac{\sqrt3}{2\pi}\log u+\frac{\widetilde A(u)}{N(u)}\right)
\;=\;\frac{3i}{2\pi}\,\log u\;+\;\widetilde h(u),
\qquad \widetilde h:=-i\sqrt3\,\frac{\widetilde A}{N}\ \text{holomorphic at }u=0.
\]
Continuation along $\ell_0$ sends $\log u\mapsto\log u+2\pi i$ and fixes $\widetilde h$, hence $1/\tau\mapsto 1/\tau+\tfrac{3i}{2\pi}\cdot 2\pi i=1/\tau-3$, i.e.\ $\tau\mapsto\tau/(1-3\tau)$.

The projective monodromy group is generated by these two transformations, so $\overline M=\langle L^{-1},T\rangle=\langle T,L\rangle$, a subgroup of $\overline{\Gamma_0(3)}$ since $T,L\in\Gamma_0(3)$. Finally, since $\gamma_0\gamma_1$ is freely homotopic to a loop encircling $u=\infty$, a suitable ordering of the standard generators of $\pi_1(\mathbb P^1\setminus\{0,1,\infty\})$ presents the local monodromy at $\infty$ as a conjugate of $(L^{-1}T)^{\pm1}$. Now $M:=L^{-1}T=\left(\begin{smallmatrix}1&1\\-3&-2\end{smallmatrix}\right)$ has trace $-1$ and determinant $1$, so the Cayley--Hamilton theorem gives $M^2+M+I=0$ and hence $M^3=I$: it is elliptic of order $3$ in $\mathrm{PSL}_2(\mathbb R)$, in agreement with $\lambda_\infty=\tfrac13$; its fixed points solve $3\tau^2+3\tau+1=0$, i.e.\ $\tau=(-3\pm i\sqrt3)/6$, the one in $\mathbb H$ being $\rho=\tau(\infty)$. (The opposite product $TL^{-1}$ also has trace $-1$ but fixes $\rho+1$, not $\rho$; the value $\tau(\infty)=\rho$ thus distinguishes the two orderings.)
\end{proof}

\begin{theorem}[The monodromy group is $\Gamma_0(3)$]
\label{thm:monodromy-Gamma03}
$\overline M=\overline{\Gamma_0(3)}$.
\end{theorem}

\begin{proof}
Let $\mathcal T^{*}:=\{\tau:-\overline\tau\in\mathcal T\}$ be the reflection of $\mathcal T$ in the imaginary axis, and set $\mathcal F:=\mathcal T\cup i\mathbb R_{>0}\cup\mathcal T^{*}$; explicitly
\[
\mathcal F=\Bigl\{\tau\in\mathbb H:\ |\Re\tau|<\tfrac12,\ \bigl|\tau+\tfrac13\bigr|>\tfrac13,\ \bigl|\tau-\tfrac13\bigr|>\tfrac13\Bigr\}.
\]
All three angles of $\mathcal T$ are of the form $\pi/m$ with $m\in\mathbb Z_{\ge2}\cup\{\infty\}$ (namely $m=\infty,\infty,3$), so by Schwarz's reflection principle the group generated by the reflections in the three sides of $\mathcal T$ is discrete with fundamental domain $\mathcal T$, and $\overline M$ is its subgroup of orientation-preserving elements, of index $2$, with fundamental domain $\mathcal F$~\cite{BeukersNotes,Yoshida}. By the Gauss--Bonnet formula---the hyperbolic area of a geodesic triangle equals $\pi$ minus the sum of its interior angles---the area of a triangle with angles $0,0,\pi/3$ is $\pi-\pi/3=2\pi/3$, so
\[
\operatorname{area}\bigl(\overline M\backslash\mathbb H\bigr)=\operatorname{area}(\mathcal F)=2\cdot\frac{2\pi}{3}=\frac{4\pi}{3}.
\]
On the other hand $[\mathrm{PSL}_2(\mathbb Z):\overline{\Gamma_0(3)}]=[\mathrm{SL}_2(\mathbb Z):\Gamma_0(3)]=4$ by the index formula $[\mathrm{SL}_2(\mathbb Z):\Gamma_0(N)]=N\prod_{p\mid N}(1+1/p)$~\cite[Ex.~1.2.3(e)]{Diamond}, while integrating the hyperbolic measure $d\mu=dx\,dy/y^2$~\cite[\S5.4]{Diamond} over the standard fundamental domain $\{|\Re\tau|\le\tfrac12,\ |\tau|\ge1\}$ gives
\[
\operatorname{area}(\mathrm{PSL}_2(\mathbb Z)\backslash\mathbb H)\;=\;\int_{-1/2}^{1/2}\frac{dx}{\sqrt{1-x^2}}\;=\;\frac{\pi}{3}.
\]
Hence
\[
\operatorname{area}\bigl(\overline{\Gamma_0(3)}\backslash\mathbb H\bigr)=4\cdot\frac{\pi}{3}=\frac{4\pi}{3}.
\]
For a subgroup $H\subseteq G$ of Fuchsian groups of finite covolume one has $\operatorname{area}(H\backslash\mathbb H)=[G:H]\cdot\operatorname{area}(G\backslash\mathbb H)$. Applying this to $\overline M\subseteq\overline{\Gamma_0(3)}$, the two covolumes are finite and equal, so $[\overline{\Gamma_0(3)}:\overline M]=1$, that is, $\overline M=\overline{\Gamma_0(3)}$.
\end{proof}

Theorem~\ref{thm:monodromy-Gamma03} in particular exhibits $\overline{\Gamma_0(3)}$ as the $(3,\infty,\infty)$ \emph{triangle group}---the orientation-preserving subgroup of the group generated by the reflections in the sides of a hyperbolic triangle with angles $\pi/3,0,0$---with $\mathcal F$ as fundamental domain and $T,L$ as side pairings. The isomorphism also follows from Takeuchi's classification of arithmetic triangle groups~\cite{Takeuchi}; we have included the direct proof to keep the treatment self-contained.

\begin{remark}[Uniqueness of the normalization]
\label{rem:normalization}
The factor $1/\sqrt3$ in Definition~\ref{def:tau} is the unique scaling for which the monodromy lies in $\Gamma_0(3)$. Indeed, repeating the computation of Proposition~\ref{prop:monodromy-group} for $\tau_\alpha:=i\alpha\,N/D$ with $\alpha>0$ gives local monodromies
\[
\begin{pmatrix}1&\alpha\sqrt3\\ 0&1\end{pmatrix}\quad\text{at }u=1,
\qquad
\begin{pmatrix}1&0\\ -\sqrt3/\alpha&1\end{pmatrix}\quad\text{at }u=0.
\]
Requiring both to lie in $\mathrm{SL}_2(\mathbb Z)$ forces $\alpha\sqrt3=m\in\mathbb Z$ and $\sqrt3/\alpha=n\in\mathbb Z$, whence $mn=3$; requiring the level condition $3\mid n$ then forces $(m,n)=(1,3)$ and $\alpha=1/\sqrt3$. In words: the product of the cusp width at $i\infty$ and the level equals $3$. No purely analytic condition on $\tau$ can pin $\alpha$ down, since conditions such as ``$\tau$ maps $(0,1)$ into $i\mathbb R_{>0}$, with $\tau\to0$ as $u\to0$ and $\tau\to i\infty$ as $u\to1$'' are invariant under $\tau\mapsto\alpha\tau$.

Conversely, the arithmetic condition \emph{does} characterize $\tau$. Suppose $\tau'$ is any Schwarz map of~\eqref{eq:hgeo-ODE} taking values in $\mathbb H$, whose projective monodromy group equals $\overline{\Gamma_0(3)}$ (not merely a conjugate of it), and which satisfies $\tau'(0)=0$, $\tau'(1)=i\infty$. Writing $\tau'=M\circ\tau$ with $M\in\mathrm{PSL}_2(\mathbb C)$, the cusp conditions force $M$ to fix $0$ and $\infty$, so $M\tau=\lambda\tau$; conjugation preserves the monodromy group, so $MTM^{-1}\colon\tau\mapsto\tau+\lambda$ and $MLM^{-1}=\left(\begin{smallmatrix}1&0\\3/\lambda&1\end{smallmatrix}\right)$ both lie in $\overline{\Gamma_0(3)}$, forcing $\lambda,1/\lambda\in\mathbb Z$, i.e.\ $\lambda=\pm1$; and $\lambda=1$ because $\tau'=\lambda\tau$ takes values in $\mathbb H$, whereas $-\tau(\mathbb H_u)=-\mathcal T$ lies in the lower half-plane. Hence $\tau'=\tau$.
\end{remark}

\subsection{The hauptmoduln \texorpdfstring{$u(\tau)$ and $v(\tau)$}{u(tau) and v(tau)}}

\begin{definition}[The modular functions $u(\tau)$ and $v(\tau)$]
\label{def:u-tau}
Let $\mathfrak u:\mathcal T\to\mathbb H_u$ be the inverse of the conformal bijection of Proposition~\ref{prop:schwarz-triangle}. Since $\mathfrak u$ is real-valued on the side $i\mathbb R_{>0}$, with $\mathfrak u\bigl(i\mathbb R_{>0}\bigr)=(0,1)\subset\mathbb R$, the Schwarz reflection principle extends $\mathfrak u$ across the side $\Re\tau=0$ to a holomorphic function on $\mathcal T\cup i\mathbb R_{>0}\cup\mathcal T^{*}=\mathcal F$, given on $\mathcal T^{*}$ by $\mathfrak u(\tau)=\overline{\mathfrak u(-\overline\tau)}$. Repeated reflection in the remaining sides extends $\mathfrak u$ to a meromorphic function on all of $\mathbb H$, invariant under the group $\overline M=\overline{\Gamma_0(3)}$ generated by pairs of such reflections. We write $u(\tau):=\mathfrak u(\tau)$ and $v(\tau):=1-u(\tau)$.
\end{definition}

The extension is well defined: $\mathfrak u$ is real-valued on each side of $\mathcal T$ (with values in $(0,1)$, $(1,\infty)$ and $(-\infty,0)$ respectively, by Proposition~\ref{prop:schwarz-triangle}), so each reflection is an instance of the Schwarz reflection principle; and the reflected copies close up around the vertex $\rho$ because the angle there is $\pi/3$, an integer submultiple of $\pi$, so that six copies of $\mathcal T$ tile a neighborhood of $\rho$.

Recall that if $\Gamma\subset\mathrm{PSL}_2(\mathbb R)$ is a Fuchsian group of finite covolume whose compactified quotient $X(\Gamma)$---the surface $\Gamma\backslash\mathbb H$ with the finitely many cusps adjoined---has genus zero, a \emph{hauptmodul} for $\Gamma$ is a $\Gamma$-invariant meromorphic function inducing a biholomorphism $X(\Gamma)\xrightarrow{\ \sim\ }\mathbb P^1(\mathbb C)$; equivalently, a $\Gamma$-invariant meromorphic function of degree $1$ on $X(\Gamma)$. A nonconstant invariant function of higher degree---the square of a hauptmodul, say---is \emph{not} a hauptmodul, so the degree assertion is the substance of the next statement.

\begin{corollary}[Hauptmoduln]
\label{cor:u-v-hauptmodul}
The functions $u(\tau)$ and $v(\tau)=1-u(\tau)$ of Definition~\ref{def:u-tau} are hauptmoduln for $\Gamma_0(3)$.
\end{corollary}

\begin{proof}
By construction $u$ is $\Gamma_0(3)$-invariant and meromorphic on $\mathbb H$, and by Proposition~\ref{prop:schwarz-triangle} it extends continuously to the cusps with $u(i\infty)=1$ and $u(0)=0$. It therefore descends to a nonconstant holomorphic map $X_0(3)=\overline{\Gamma_0(3)}\backslash\mathbb H^{*}\to\mathbb P^1$, where $\mathbb H^{*}=\mathbb H\cup\mathbb Q\cup\{\infty\}$. We show that this map is injective; an injective holomorphic map between compact Riemann surfaces has degree $1$.

By Theorem~\ref{thm:monodromy-Gamma03} the region $\mathcal F$ is a fundamental domain for $\Gamma_0(3)$, and its interior is the disjoint union $\mathcal T\sqcup i\mathbb R_{>0}\sqcup\mathcal T^{*}$. Proposition~\ref{prop:schwarz-triangle} says that $u$ maps $\mathcal T$ bijectively onto $\mathbb H_u$; the reflection defining $u$ on $\mathcal T^{*}$ then maps $\mathcal T^{*}$ bijectively onto the lower half-plane, and $i\mathbb R_{>0}$ bijectively onto $(0,1)$. These three images are pairwise disjoint, so $u$ is injective on the interior of $\mathcal F$. The boundary of $\mathcal F$ in $\mathbb H$ consists of the two vertical sides $\Re\tau=\pm\tfrac12$, interchanged by $T$, and the two arcs $|\tau\mp\tfrac13|=\tfrac13$, interchanged by $L$ (indeed $L(0)=0$ and $L\rho=\rho+1$); these are exactly the identifications defining $X_0(3)$, and $u$ respects them by invariance. The elliptic point $\rho$ is the unique point of $X_0(3)$ mapping to $\infty$, and the cusps $0$ and $i\infty$ are the unique points mapping to $0$ and to $1$. Hence $u$ induces a bijection $X_0(3)\to\mathbb P^1$, so $\deg u=1$ and $u$ is a hauptmodul. The same holds for $v=1-u$.
\end{proof}

\subsection{The Fricke symmetry}

The hypergeometric equation~\eqref{eq:hgeo-ODE} is manifestly invariant under $u\mapsto 1-u$. This elementary symmetry is the source of the Fricke involution, and hence of the functional equation in Theorem~\ref{thm:uniqueness}.

\begin{lemma}[Fricke symmetry of the hauptmodul]
\label{lem:fricke-u}
For all $\tau\in\mathbb H$,
\[
u\!\left(-\frac{1}{3\tau}\right)\;=\;1-u(\tau),\qquad\text{equivalently}\qquad v\circ W_3=u.
\]
\end{lemma}

\begin{proof}
The substitution $w\mapsto1-w$ in Definition~\ref{def:AB} shows that $N(1-u)=D(u)$ and $D(1-u)=N(u)$ for $u\in(0,1)$; hence $s:=N/D$ satisfies $s(1-u)=1/s(u)$ there. Writing $\tau=is/\sqrt3$, so that $s=-i\sqrt3\,\tau$, this becomes
\[
\tau(1-u)\;=\;\frac{i}{\sqrt3}\cdot\frac{1}{s(u)}\;=\;\frac{i}{\sqrt3}\cdot\frac{1}{-i\sqrt3\,\tau(u)}\;=\;-\frac{1}{3\tau(u)}\;=\;W_3\,\tau(u).
\]
Applying $\mathfrak u$ and using $\mathfrak u(\tau(u))=u$ gives $\mathfrak u(W_3\tau)=1-\mathfrak u(\tau)$ for every $\tau\in i\mathbb R_{>0}$. Both sides are meromorphic on $\mathbb H$ (Definition~\ref{def:u-tau}, together with $W_3\mathbb H=\mathbb H$) and agree on the imaginary axis, which has limit points in $\mathbb H$; by the identity theorem they agree on all of $\mathbb H$.
\end{proof}

The Fricke symmetry of the crossing probability, which is hypothesis~(ii) of Theorem~\ref{thm:uniqueness}, is now immediate and requires no modular input at all.

\begin{corollary}[Fricke symmetry of the crossing probability]
\label{cor:fricke-Pi}
For every $r>0$ one has $\Pi_h^{\mathrm{par}}(1/r)=1-\Pi_h^{\mathrm{par}}(r)$. Equivalently, writing $\widetilde\Pi_h^{\mathrm{par}}(\tau):=\Pi_h^{\mathrm{par}}(-i\sqrt3\,\tau)$ for $\tau\in i\mathbb R_{>0}$,
\[
\widetilde\Pi_h^{\mathrm{par}}(W_3\tau)\;=\;1-\widetilde\Pi_h^{\mathrm{par}}(\tau).
\]
\end{corollary}

\begin{proof}
By Theorem~\ref{thm:r-ratio}, $r(1-u)=N(1-u)/D(1-u)=D(u)/N(u)=1/r(u)$, so the prevertex of $P_{1/r}$ is $1-u(r)$. The integrand $w^{-2/3}(1-w)^{-2/3}$ is invariant under $w\mapsto1-w$, whence $\int_{1-u}^1=\int_0^{u}$ and therefore, by Theorem~\ref{thm:cardy-parallelogram}(i),
\[
\Pi_h^{\mathrm{par}}(1/r)+\Pi_h^{\mathrm{par}}(r)\;=\;\frac{1}{B(\tfrac13,\tfrac13)}\left(\int_0^{u}+\int_{u}^1\right)w^{-2/3}(1-w)^{-2/3}\,dw\;=\;1 .
\]
The second form follows from Lemma~\ref{lem:fricke-u}, since $\tau=ir/\sqrt3$ turns $r\mapsto1/r$ into $\tau\mapsto-1/(3\tau)$.
\end{proof}

Geometrically, rescaling $P_r$ by $1/r$ produces a $\pi/3$ parallelogram congruent to $P_{1/r}$ with the roles of the two pairs of opposite sides interchanged; the horizontal crossing of $P_{1/r}$ therefore corresponds to the crossing of $P_r$ between the other pair of sides, whose probability is $1-\Pi_h^{\mathrm{par}}(r)$. Only at the fixed point $r=1$ is $P_r$ a rhombus, self-similar under this exchange; there $\tau=i/\sqrt3$, $v=\tfrac12$ and $\Pi_h^{\mathrm{par}}=\tfrac12$.

\subsection{The standard hauptmodul and the eta quotient}

A classical hauptmodul for $\Gamma_0(3)$ is the eta quotient
\[
t_3(\tau):=\left(\frac{\eta(\tau)}{\eta(3\tau)}\right)^{12}=q^{-1}-12+54q-\cdots,\qquad q=e^{2\pi i\tau},
\]
which is holomorphic and nonvanishing on $\mathbb H$, has a simple pole at the cusp $i\infty$ and a simple zero at the cusp $0$; the eta-quotient hauptmoduln for the genus-zero groups $\Gamma_0(N)$ are classical, see~\cite[\S3.1]{Maier} or~\cite{Köhler}. The Fricke involution acts on it by
\begin{equation}
\label{eq:t3-fricke}
t_3\!\left(-\frac{1}{3\tau}\right)\;=\;\frac{3^6}{t_3(\tau)},
\end{equation}
which follows at once from $\eta(-1/\sigma)=\sqrt{-i\sigma}\,\eta(\sigma)$ applied with $\sigma=3\tau$ and with $\sigma=\tau$.

\begin{lemma}[Explicit form of the hauptmoduln]
\label{lem:u-v-explicit}
For the hauptmoduln $u(\tau)$, $v(\tau)=1-u(\tau)$ of Definition~\ref{def:u-tau},
\[
u(\tau)\;=\;\frac{t_3(\tau)}{t_3(\tau)+27},\qquad
v(\tau)\;=\;\frac{27}{t_3(\tau)+27}\;=\;\frac{27\,\eta(3\tau)^{12}}{\eta(\tau)^{12}+27\,\eta(3\tau)^{12}}.
\]
\end{lemma}

\begin{proof}
Both $u$ and $t_3$ are hauptmoduln for $\Gamma_0(3)$ (Corollary~\ref{cor:u-v-hauptmodul}), so they are related by a M\"obius transformation: $u=(at_3+b)/(ct_3+d)$ with $ad-bc\ne0$. As $\tau\to i\infty$ we have $t_3\to\infty$ and, by Proposition~\ref{prop:schwarz-triangle}, $u\to1$; hence $a/c=1$. As $\tau\to0$ we have $t_3\to0$ by~\eqref{eq:t3-fricke} and $u\to0$; hence $b/d=0$, i.e.\ $b=0$. Therefore
\[
u=\frac{t_3}{t_3+k},\qquad k:=d/c\ne0 .
\]
It remains to determine $k$; for this we use the Fricke symmetry. Combining Lemma~\ref{lem:fricke-u} with~\eqref{eq:t3-fricke},
\[
1-\frac{t_3}{t_3+k}\;=\;1-u(\tau)\;=\;u(W_3\tau)\;=\;\frac{3^6/t_3}{3^6/t_3+k}\;=\;\frac{3^6}{3^6+k\,t_3},
\]
that is $k(3^6+k\,t_3)=3^6(t_3+k)$ identically in $\tau$, whence $k^2=3^6$ and $k=\pm27$. Since $u$ is real and lies in $(0,1)$ on the imaginary axis, where $t_3>0$, we must have $k>0$; thus $k=27$. The final expression follows from the definition of $t_3$.
\end{proof}

Because $\tau=ir/\sqrt3$ on the imaginary axis (Definition~\ref{def:tau}), Lemma~\ref{lem:u-v-explicit} identifies the geometric prevertex $u(r)$ of \S\ref{sec:SC} with the eta quotient~\eqref{eq:v-eta}, which is assertion~(ii) of Theorem~\ref{thm:cardy-parallelogram}. Together with the computation in \S\ref{subsec:hypergeometric} this completes the proof of that theorem.

\begin{remark}[The constant $27$]
\label{rem:27}
The constant $27$ admits a second, independent derivation from the cusp asymptotics~\eqref{eq:cusp-asymptotics}, which give $1-u=\varepsilon e^{i(\theta-\pi)}=27\,e^{2\pi i\tau}+O(q^2)$; that is, $v=27q+O(q^2)$, in agreement with $v=27/(t_3+27)$ and $t_3=q^{-1}+O(1)$. In this derivation the constant arises as $27=\exp\bigl(2\psi(1)-\psi(\tfrac13)-\psi(\tfrac23)\bigr)=e^{3\log 3}$, a digamma evaluation.
\end{remark}

\subsection{Cubic theta functions}

To pass from the hauptmodul to the eta \emph{product} $\eta(\tau)^2\eta(3\tau)^2$ in \S\ref{sec:proof-main} we use the cubic theta functions of Borwein and Borwein~\cite{Borwein}.

\begin{definition}[Borwein cubic theta functions]
\label{def:borwein-theta}
For $\tau\in\mathbb H$ and $q=e^{2\pi i\tau}$, set
\begin{align*}
a(\tau) &\;=\;\sum_{m,n\in\mathbb Z} q^{m^2+mn+n^2}, \\
b(\tau) &\;=\;\sum_{m,n\in\mathbb Z} \omega^{m-n}\,q^{m^2+mn+n^2}\qquad (\omega = e^{2\pi i/3}), \\
c(\tau) &\;=\;\sum_{m,n\in\mathbb Z} q^{(m+1/3)^2+(m+1/3)(n+1/3)+(n+1/3)^2}.
\end{align*}
\end{definition}

\begin{lemma}[Borwein identities]
\label{lem:borwein-identities}
The cubic theta functions satisfy the AGM-type identity
\begin{equation}
\label{eq:borwein-cubic}
a(\tau)^3 \;=\; b(\tau)^3 + c(\tau)^3,
\end{equation}
the eta-quotient expressions
\begin{equation}
\label{eq:borwein-eta-quot}
b(\tau) \;=\; \frac{\eta(\tau)^3}{\eta(3\tau)}, \qquad c(\tau) \;=\; 3\,\frac{\eta(3\tau)^3}{\eta(\tau)},
\end{equation}
and the hypergeometric identity
\begin{equation}
\label{eq:borwein-hyp}
G\!\left(\frac{c(\tau)^3}{a(\tau)^3}\right) \;=\; a(\tau),\qquad G(z)={}_2F_1(\tfrac13,\tfrac23;1;z).
\end{equation}
In particular $b(\tau)\,c(\tau)=3\,\eta(\tau)^2\eta(3\tau)^2$.
\end{lemma}

\begin{proof}
In~\cite{Borwein} the functions $a,b,c$ are defined through the single theta series $L(q)=\sum_{m,n}q^{m^2+mn+n^2}$ by $a=L(q)$, $b=[3L(q^3)-L(q)]/2$, $c=[L(q^{1/3})-L(q)]/2$, and the series representations of Definition~\ref{def:borwein-theta} are derived in the proof of Theorem~2.3 there. With those definitions, the cubic identity~\eqref{eq:borwein-cubic} is~\cite[Eq.~(2.3)]{Borwein} (the ``basic cubic identity''), and the hypergeometric identity~\eqref{eq:borwein-hyp} is~\cite[Thm.~2.3(a)]{Borwein}. The eta-quotient expressions~\eqref{eq:borwein-eta-quot} do not appear in~\cite{Borwein}; see~\cite{BBG} or~\cite{Köhler}. The last assertion is the product of the two formulas in~\eqref{eq:borwein-eta-quot}.
\end{proof}

\begin{lemma}[Cubic-theta expression for $u,v$]
\label{lem:u-v-theta}
For $u(\tau)$ and $v(\tau)$ as in Definition~\ref{def:u-tau}, the identities
\[
u(\tau)=\frac{b(\tau)^3}{a(\tau)^3},\qquad v(\tau)=\frac{c(\tau)^3}{a(\tau)^3}
\]
hold on all of $\mathbb H$, and
\[
G\bigl(v(\tau)\bigr)=a(\tau)\qquad\text{for }\tau\in i\mathbb R_{>0},
\]
where both sides are real-analytic and positive (on the axis $v(\tau)\in(0,1)$, so the principal branch of $G$ applies).
\end{lemma}

\begin{proof}
By~\eqref{eq:borwein-eta-quot}, $b^3/c^3=\eta(\tau)^{12}/\bigl(27\,\eta(3\tau)^{12}\bigr)=t_3/27$. Hence, using~\eqref{eq:borwein-cubic},
\[
\frac{c^3}{a^3}=\frac{c^3}{b^3+c^3}=\frac{1}{1+b^3/c^3}=\frac{27}{t_3+27}=v(\tau)
\]
by Lemma~\ref{lem:u-v-explicit}; both sides are meromorphic on $\mathbb H$, so the identity holds there, and consequently $b^3/a^3=1-v=u$. The final identity is~\eqref{eq:borwein-hyp}, valid wherever $c^3/a^3$ lies in the disc of convergence and extended along the axis by real-analytic continuation: on $i\mathbb R_{>0}$ one has $v\in(0,1)$ by Theorem~\ref{thm:cardy-parallelogram} and Lemma~\ref{lem:u-v-explicit}, and $a>0$, so both sides are positive real-analytic functions there. (For $\tau$ away from the axis, $v(\tau)$ may cross the branch cut $[1,\infty)$ of $G$ and the identity is then only local; only the identity on the axis is used in the sequel.)
\end{proof}

\section{Proof of Theorem~\ref{thm:main-almost-modular}}
\label{sec:proof-main}

With explicit formulas for the hauptmoduln $u, v$ in the uniformizing coordinate $\tau$ of Definition~\ref{def:tau}, the proof of Theorem~\ref{thm:main-almost-modular} reduces to differentiating $\Pi_h^{\mathrm{par}}$ in $\tau$ and identifying the result.

Recall from \S\ref{sec:monodromy} the notation $G(z)={}_2F_1(\tfrac13,\tfrac23;1;z)$ and $B=B(\tfrac13,\tfrac23)=2\pi/\sqrt3$, so that the period integrals of Definition~\ref{def:AB} satisfy $N(u)=B\,G(u)$ and $D(u)=B\,G(1-u)$.

\begin{lemma}[Derivative of the hauptmodul]
\label{lem:dvdtau}
For $\tau\in i\mathbb R_{>0}$, where $v=v(\tau)\in(0,1)$, and with $G$ as above,
\[
\frac{dv}{d\tau} \;=\; 2\pi i\,v(1-v)\,G(v)^2.
\]
\end{lemma}

\begin{proof}
From $\tau(u) = \tfrac{i}{\sqrt 3}\,N(u)/D(u)$,
\[
\frac{d\tau}{du} \;=\; \frac{i}{\sqrt 3}\,\frac{W(u)}{D(u)^2}, \qquad W(u) \;:=\; N'(u)\,D(u) - N(u)\,D'(u).
\]
Both $N$ and $D$ solve~\eqref{eq:hgeo-ODE}, which in normalized form reads $F''+p\,F'+q\,F=0$ with $p(u)=(1-2u)/(u(1-u))$. The Abel--Liouville formula for the Wronskian of two solutions of such an equation, namely $W'=-p\,W$, gives $W'/W=-(1-2u)/(u(1-u))$ and hence
\[
W(u)\;=\;\frac{W_0}{u(1-u)}\qquad\text{for a constant }W_0 .
\]
To evaluate $W_0$ we let $u\to1^-$ and use the connection formula of Lemma~\ref{lem:connection}, $N=-\tfrac{\sqrt3}{2\pi}D\log(1-u)+A$ with $A$ holomorphic at $u=1$. Differentiating,
\[
N'=\frac{\sqrt3}{2\pi}\,\frac{D}{1-u}-\frac{\sqrt3}{2\pi}\,D'\log(1-u)+A',
\]
so that in $W=N'D-ND'$ the two logarithmic terms cancel and
\[
W(u)\;=\;\frac{\sqrt3}{2\pi}\cdot\frac{D(u)^2}{1-u}\;+\;A'(u)D(u)-A(u)D'(u).
\]
The last two terms are bounded near $u=1$, while $D(1)=B$; multiplying by $u(1-u)$ and letting $u\to1^-$ therefore gives
\[
W_0\;=\;\frac{\sqrt3}{2\pi}\,B^2\;=\;\frac{\sqrt3}{2\pi}\cdot\frac{4\pi^2}{3}\;=\;\frac{2\pi}{\sqrt3}\;=\;B ,
\]
using the Euler reflection identity $B=B(\tfrac13,\tfrac23)=\Gamma(\tfrac13)\Gamma(\tfrac23)=\pi/\sin(\pi/3)=2\pi/\sqrt3$ and hence $B^2=4\pi^2/3$. (In particular $W_0>0$, consistent with the fact that $\tau=ir(u)/\sqrt3$ ascends the imaginary axis as $u$ increases, $r$ being increasing by Lemma~\ref{lem:endpoint-r}.) Substituting $D(u)=B\,G(1-u)$,
\[
\frac{d\tau}{du} \;=\; \frac{i}{\sqrt 3}\cdot\frac{W_0}{u(1-u)\,D(u)^2} \;=\; \frac{i}{2\pi}\cdot\frac{1}{u(1-u)\,G(1-u)^2}.
\]
Since $dv/du = -1$, $u(1-u)=v(1-v)$ and $G(1-u) = G(v)$, the chain rule gives $dv/d\tau=-\bigl(d\tau/du\bigr)^{-1}=2\pi i\,v(1-v)G(v)^2$.
\end{proof}

Part~(ii) of the next theorem uses the following standard language. For $\gamma=\left(\begin{smallmatrix}a&b\\c&d\end{smallmatrix}\right)\in\mathrm{GL}_2^+(\mathbb R)$ and $k\in\mathbb Z$, the \emph{weight-$k$ slash operator} is
\[
(g|_k\gamma)(\tau)\;:=\;\det(\gamma)^{k/2}\,(c\tau+d)^{-k}\,g(\gamma\tau).
\]
A holomorphic function $g$ on $\mathbb H$ is a \emph{weight-$k$ cusp form on $\Gamma_0(3)$ with multiplier system $v$}---where $v:\Gamma_0(3)\to\mathbb C^\times$ is a map of modulus one such that $g|_k\gamma=v(\gamma)\,g$ for all $\gamma\in\Gamma_0(3)$---if in addition $g$ vanishes at the two cusps of $\Gamma_0(3)$, meaning $g(\tau)\to0$ and $(g|_kW_3)(\tau)\to0$ as $\Im\tau\to\infty$. The multiplier system $v$ has \emph{order $3$} if $v^3\equiv1$.

\begin{theorem}[Modular formula for the derivative]
\label{thm:modular-derivative}
Write $\widetilde\Pi_h^{\mathrm{par}}(\tau) := \Pi_h^{\mathrm{par}}(-i\sqrt 3\,\tau)$ for $\tau\in i\mathbb R_{>0}$, as in Corollary~\ref{cor:fricke-Pi}. Then the following hold.
\begin{enumerate}
\item[\textup{(i)}] $\widetilde\Pi_h^{\mathrm{par}}$ extends to a holomorphic function on $\mathbb H$, again denoted $\widetilde\Pi_h^{\mathrm{par}}$, whose derivative is a constant multiple of the eta product $\eta(\tau)^2\eta(3\tau)^2$:
\[
\frac{d\widetilde\Pi_h^{\mathrm{par}}}{d\tau} \;=\; c_\eta\,\eta(\tau)^2\eta(3\tau)^2 \quad\text{on }\mathbb H, \qquad c_\eta \;=\; \frac{12\pi^2 i}{\sqrt 3\,\Gamma(\tfrac13)^3}.
\]
\item[\textup{(ii)}] The function $\eta(\tau)^2\eta(3\tau)^2$ is a weight-$2$ cusp form on $\Gamma_0(3)$ with respect to a multiplier system $v_\eta$ of order $3$, with $v_\eta(T) = e^{2\pi i/3}$; under the Fricke involution it satisfies
\[
\bigl(\eta^2\eta(3\cdot)^2\bigr)\big|_2 W_3 \;=\; -\,\eta^2\eta(3\cdot)^2,\qquad\text{i.e.}\qquad v_\eta(W_3)=-1.
\]
\end{enumerate}
In particular, $\widetilde\Pi_h^{\mathrm{par}}(\tau)=c_\eta\int_{i\infty}^{\tau}\eta(z)^2\eta(3z)^2\,dz$: up to a constant, the Cardy--Smirnov function is the integral of a weight-$2$ cusp form, an integral of Eichler type.
\end{theorem}

\begin{proof}
Work first on the axis $\tau\in i\mathbb R_{>0}$, where $u,v\in(0,1)$ and $a,b,c>0$. By Theorem~\ref{thm:cardy-parallelogram}(i), $\Pi_h^{\mathrm{par}}(r) = \frac{1}{B(\tfrac13,\tfrac13)} \int_u^1 w^{-2/3}(1-w)^{-2/3}\,dw$ with $v = 1 - u$. Differentiating with respect to $v$ (using $du = -dv$),
\[
\frac{d\Pi_h^{\mathrm{par}}}{dv} \;=\; \frac{1}{B(\tfrac13,\tfrac13)}\,v^{-2/3}(1-v)^{-2/3}.
\]
By the chain rule and Lemma~\ref{lem:dvdtau},
\[
\frac{d\widetilde\Pi_h^{\mathrm{par}}}{d\tau} \;=\; \frac{d\Pi_h^{\mathrm{par}}}{dv}\cdot\frac{dv}{d\tau} \;=\; \frac{2\pi i}{B(\tfrac13,\tfrac13)}\,v^{1/3}(1-v)^{1/3}\,G(v)^2\qquad(\tau\in i\mathbb R_{>0}).
\]
Substituting the cubic-theta expressions $u = b(\tau)^3/a(\tau)^3$, $v = c(\tau)^3/a(\tau)^3$, $G(v) = a(\tau)$ from Lemma~\ref{lem:u-v-theta}, and using positivity on the axis to take cube roots,
\[
v^{1/3}(1-v)^{1/3}\,G(v)^2 \;=\; \frac{c(\tau)}{a(\tau)}\cdot\frac{b(\tau)}{a(\tau)}\cdot a(\tau)^2 \;=\; b(\tau)\,c(\tau).
\]
By Lemma~\ref{lem:borwein-identities}, $b(\tau)\,c(\tau) = 3\,\eta(\tau)^2\eta(3\tau)^2$. Hence
\[
\frac{d\widetilde\Pi_h^{\mathrm{par}}}{d\tau} \;=\; \frac{6\pi i}{B(\tfrac13,\tfrac13)}\,\eta(\tau)^2\eta(3\tau)^2\qquad\text{on }i\mathbb R_{>0}.
\]
By the Euler reflection formula $\Gamma(\tfrac13)\Gamma(\tfrac23) = \pi/\sin(\pi/3) = 2\pi/\sqrt 3$,
\[
B(\tfrac13,\tfrac13) \;=\; \frac{\Gamma(\tfrac13)^2}{\Gamma(\tfrac23)} \;=\; \frac{\Gamma(\tfrac13)^3}{\Gamma(\tfrac13)\Gamma(\tfrac23)} \;=\; \frac{\sqrt 3\,\Gamma(\tfrac13)^3}{2\pi},
\]
so $1/B(\tfrac13,\tfrac13) = 2\pi/(\sqrt 3\,\Gamma(\tfrac13)^3)$ and $c_\eta = 6\pi i/B(\tfrac13,\tfrac13) = 12\pi^2 i/(\sqrt 3\,\Gamma(\tfrac13)^3)$.

For the extension, set $F(\tau):=c_\eta\int_{i\infty}^{\tau}\eta(z)^2\eta(3z)^2\,dz$, the integral being taken along any path in $\mathbb H$; since the integrand is holomorphic on $\mathbb H$ and decays like $q^{1/3}$ at the cusp, $F$ is a well-defined holomorphic function on $\mathbb H$ vanishing as $\Im\tau\to\infty$. On the axis, $F$ and $\widetilde\Pi_h^{\mathrm{par}}$ have equal derivatives by the identity just proved, and both tend to $0$ at $i\infty$: for $F$ this is immediate, and for $\widetilde\Pi_h^{\mathrm{par}}$ it follows since $u(r)\to1$ as $r\to\infty$ (Lemma~\ref{lem:endpoint-r}), so that $\Pi_h^{\mathrm{par}}(r)=\frac{1}{B(1/3,1/3)}\int_{u(r)}^1 w^{-2/3}(1-w)^{-2/3}\,dw\to0$. Hence the two functions agree on $i\mathbb R_{>0}$, and $F$ is the asserted holomorphic extension.

The claims in~(ii) follow from the $\eta$ transformation rules (see~\cite[\S1.2]{Diamond}). Under $T$,
\[
\eta(\tau)^2\eta(3\tau)^2 \longmapsto e^{2\pi i/12}\eta(\tau)^2 \cdot e^{6\pi i/12}\eta(3\tau)^2 = e^{2\pi i/3}\eta(\tau)^2\eta(3\tau)^2,
\]
giving $v_\eta(T) = e^{2\pi i/3}$. For $W_3$, the transformation $\eta(-1/\sigma) = \sqrt{-i\sigma}\,\eta(\sigma)$ (applied with $\sigma = 3\tau$ and with $\sigma = \tau$, and using $3\,W_3\tau = -1/\tau$) gives
\[
\eta(W_3\tau)^2\,\eta(3\,W_3\tau)^2 \;=\; \eta\!\left(-\tfrac{1}{3\tau}\right)^{\!2}\eta\!\left(-\tfrac{1}{\tau}\right)^{\!2} \;=\; (-3i\tau)(-i\tau)\,\eta(3\tau)^2\eta(\tau)^2 \;=\; -3\tau^2\,\eta(\tau)^2\eta(3\tau)^2,
\]
so that, applying the weight-$2$ slash operator $(g|_2 W_3)(\tau) = 3\,(3\tau)^{-2} g(W_3\tau)$,
\[
\big(\eta^2\eta(3\cdot)^2\big)\big|_2 W_3 \;=\; 3\,(3\tau)^{-2}\bigl(-3\tau^2\bigr)\eta(\tau)^2\eta(3\tau)^2 \;=\; -\,\eta(\tau)^2\eta(3\tau)^2,
\]
i.e.\ $v_\eta(W_3) = -1$.
\end{proof}

\begin{remark}[Geometric origin of the Fricke eigenvalue]
\label{rem:fricke-eigenvalue}
The eigenvalue $v_\eta(W_3)=-1$ can also be obtained without eta identities. By Corollary~\ref{cor:fricke-Pi} the identity $\widetilde\Pi_h^{\mathrm{par}}(W_3\tau)=1-\widetilde\Pi_h^{\mathrm{par}}(\tau)$ holds on $i\mathbb R_{>0}$; both sides are holomorphic on $\mathbb H$ by part~(i), so it holds on all of $\mathbb H$ by the identity theorem, and differentiating it recovers $\bigl(\eta^2\eta(3\cdot)^2\bigr)\big|_2W_3=-\,\eta^2\eta(3\cdot)^2$ directly. The minus sign thus reflects the geometric fact that rescaling $P_r$ by $1/r$ exchanges its two pairs of opposite sides.
\end{remark}

\begin{proof}[Proof of Theorem~\ref{thm:main-almost-modular}]
By Definition~\ref{def:tau}, $\tau = ir/\sqrt 3$ on the imaginary axis, so $d\tau/dr = i/\sqrt 3$ and $3\tau = i\sqrt 3\,r$. Combining with Theorem~\ref{thm:modular-derivative},
\[
\frac{d\Pi_h^{\mathrm{par}}}{dr} \;=\; \frac{d\Pi_h^{\mathrm{par}}}{d\tau}\cdot\frac{d\tau}{dr} \;=\; \frac{12\pi^2 i}{\sqrt 3\,\Gamma(\tfrac13)^3}\cdot\frac{i}{\sqrt 3}\,\eta\!\left(\tfrac{i\,r}{\sqrt 3}\right)^{\!2}\!\eta\!\left(i\sqrt 3\,r\right)^{\!2} \;=\; -\frac{4\pi^2}{\Gamma(\tfrac13)^3}\,\eta\!\left(\tfrac{i\,r}{\sqrt 3}\right)^{\!2}\!\eta\!\left(i\sqrt 3\,r\right)^{\!2}.
\]
We claim $\Pi_h^{\mathrm{par}}(r) \to 0$ as $r \to \infty$. Indeed, by Theorem~\ref{thm:cardy-parallelogram} and Lemma~\ref{lem:endpoint-r}, $\Pi_h^{\mathrm{par}}(r) = 1 - g(r) = 1 - \Phi_\Delta(u(r))$ where $u(r) \to 1$ as $r \to \infty$, so $\Phi_\Delta(u(r)) \to \Phi_\Delta(1) = 1$. Integrating the displayed identity from $r$ to $\infty$ therefore gives
\[
\Pi_h^{\mathrm{par}}(r) \;=\; -\int_r^\infty \frac{d\Pi_h^{\mathrm{par}}}{dt}\,dt \;=\; \frac{4\pi^2}{\Gamma(\tfrac13)^3}\int_r^\infty \eta\!\left(\tfrac{i\,t}{\sqrt 3}\right)^{\!2}\!\eta\!\left(i\sqrt 3\,t\right)^{\!2}\,dt,
\]
the claimed formula.
\end{proof}

\section{Uniqueness: proof of Theorem~\ref{thm:uniqueness}}
\label{sec:rigidity}

\subsection{Strategy}

The \emph{Fricke involution} at level $N$ is the matrix
\[
W_N=\begin{pmatrix}0&-1\\ N&0\end{pmatrix},\qquad\text{acting by}\quad W_N\tau=-\frac{1}{N\tau},
\]
and the \emph{Fricke extension} is
\[
\Gamma_0^+(N):=\langle\Gamma_0(N),W_N\rangle,
\]
where $\langle\,\cdot\,\rangle$ denotes the subgroup of $\mathrm{PSL}_2(\mathbb R)$ generated by the indicated elements; thus $\Gamma_0^+(N)$ is generated by $\Gamma_0(N)$ together with the Fricke involution $W_N$.
For a Fuchsian group $\Gamma\subset\mathrm{PSL}_2(\mathbb R)$ of finite covolume, commensurable with $\mathrm{PSL}_2(\mathbb Z)$, we write $M_k(\Gamma)$ for the space of holomorphic functions $g$ on $\mathbb H$ satisfying $g|_k\gamma=g$ for all $\gamma\in\Gamma$ and holomorphic at each cusp, and $S_k(\Gamma)\subseteq M_k(\Gamma)$ for the subspace vanishing at each cusp; this extends the usual definitions for congruence subgroups~\cite[\S1.2]{Diamond} to groups such as $\Gamma_0^+(3)$, which is not contained in $\mathrm{PSL}_2(\mathbb Z)$.

Let $F$ be a function satisfying hypotheses~(i)--(iii) of Theorem~\ref{thm:uniqueness}, and set $f := dF/d\tau$. The proof proceeds in four steps:
\begin{enumerate}
\item The $q$-expansion gives $f|_2 T = A\,f$ with $A := e^{2\pi i\alpha}$.
\item Differentiating the Fricke symmetry gives $f|_2 W_3 = -f$.
\item Combining (1) and (2), the symmetry $F(U\tau) = 1 - A\,F(\tau)$ for $U = W_3 T$ iterates as an affine map on $\mathbb C$. Since $U$ has finite order in $\mathrm{PSL}_2(\mathbb R)$, this affine map must also have finite order, forcing $A^6 = 1$ and hence $6\alpha \in \mathbb Z_{>0}$ with $\alpha \not\equiv \tfrac12 \pmod 1$. Steps (1) and (2) then upgrade to $f^6 \in S_{12}(\Gamma_0^+(3))$.
\item The nonvanishing hypothesis~(iii) forces $f^6$ to vanish to order exactly $2$ at the cusp, which by the valence formula pins down $\alpha = \tfrac13$, identifies $f^6$ as a scalar multiple of $\Delta_3^+ = (\eta(\tau)\eta(3\tau))^{12}$, and hence determines $f$ up to a scalar.
\end{enumerate}

The argument follows the structure of Kleban and Zagier's uniqueness proof~\cite[\S2]{KZ}, with two differences. First, the relevant group $\Gamma_0^+(3)$ is smaller than the $\mathrm{SL}_2(\mathbb Z)$ appearing in the rectangular case. Second, and more substantively, $\dim S_{12}(\Gamma_0^+(3)) = 2$ rather than $\dim S_{12}(\mathrm{SL}_2(\mathbb Z)) = 1$ (the latter spanned by $\Delta$). In the rectangular case the one-dimensionality of the cusp-form space leaves no room for an extraneous solution, and the Fricke symmetry and $q$-block ansatz suffice. At level $3$ the symmetry hypotheses fall exactly one condition short of determining the cusp form, and the deficiency is realized by a genuine second solution (Remark~\ref{rem:ghost}); the nondegeneracy condition~(iii) supplies the missing constraint.

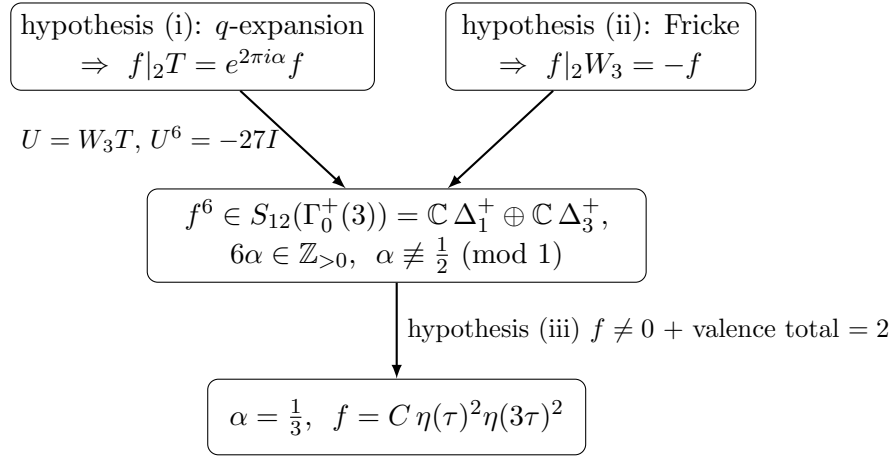
\begin{figure}[htbp]
\centering
\begin{tikzpicture}[>=latex, node distance=3.0cm]
\node (q) [draw, rounded corners, align=center, minimum width=4.1cm, minimum height=1.0cm] {hypothesis (i): $q$-expansion\\[1pt] $\Rightarrow\ f|_2T=e^{2\pi i\alpha}f$};
\node (fricke) [draw, rounded corners, align=center, right of=q, xshift=2.4cm, minimum width=4.1cm, minimum height=1.0cm] {hypothesis (ii): Fricke\\[1pt] $\Rightarrow\ f|_2W_3=-f$};
\node (f6) [draw, rounded corners, align=center, below of=q, xshift=2.7cm, yshift=0.5cm, minimum width=6.6cm, minimum height=1.0cm] {$f^6\in S_{12}(\Gamma_0^+(3))=\mathbb C\,\Delta_1^+\oplus\mathbb C\,\Delta_3^+$,\\[1pt] $6\alpha\in\mathbb Z_{>0}$, $\ \alpha\not\equiv\tfrac12\ (\mathrm{mod}\ 1)$};
\node (eta) [draw, rounded corners, align=center, below of=f6, yshift=0.6cm, minimum width=5.0cm, minimum height=1.0cm] {$\alpha=\tfrac13$, $\ f=C\,\eta(\tau)^2\eta(3\tau)^2$};
\draw[->, thick] (q) -- node[left, xshift=-1pt]{\small $U=W_3T$,\ $U^6=-27I$} (f6);
\draw[->, thick] (fricke) -- (f6);
\draw[->, thick] (f6) -- node[right]{\small hypothesis (iii) $f\neq0$\ +\ valence total $=2$} (eta);
\end{tikzpicture}
\caption{Structure of the uniqueness argument. The two symmetry hypotheses produce a weight-$12$ form on the Fricke extension; because $\dim S_{12}(\Gamma_0^+(3))=2$, the nondegeneracy hypothesis~(iii) is needed to select the line $\mathbb C\,\Delta_3^+$ and force $\alpha=\tfrac13$.}
\label{fig:rigidity-architecture}
\end{figure}

\subsection{The weight-12 cusp form space on \texorpdfstring{$\Gamma_0^+(3)$}{Gamma\_0\^+(3)}}

\begin{lemma}[Signature of the Fricke extension]
\label{lem:signature-Fricke}
$\Gamma_0^+(3)$ has genus $0$, a single cusp, one elliptic point of order $2$ and one of order $6$. In the standard signature notation $(g;m_1,\dots,m_r;s)$---genus, orders of the elliptic points, number of cusps---its signature is $(0;2,6;1)$.
\end{lemma}

\begin{proof}
Since $[\mathrm{SL}_2(\mathbb Z) : \Gamma_0(3)] = 4$ and $\Gamma_0^+(3) = \Gamma_0(3) \sqcup W_3 \Gamma_0(3)$, we have $[\mathrm{SL}_2(\mathbb Z) : \Gamma_0^+(3)] = 2$, hence $\mathrm{area}(\Gamma_0^+(3) \backslash \mathbb H) = 2\pi/3$.

For prime $N$, $\Gamma_0(N)$ has exactly two cusps $0$ and $\infty$~\cite[Ex.~3.1.4(b)]{Diamond}, which are exchanged by the Fricke involution $W_N$~\cite{AtkinLehner}; hence $\Gamma_0^+(3)$ has a single cusp. The elliptic points come from fixed points of $W_3$ and $W_3 T$. Solving $W_3 \tau = \tau$, i.e., $-1/(3\tau) = \tau$, gives $\tau_2 = i/\sqrt 3$; since $W_3^2 = -3I$ is scalar, $W_3$ has order $2$ in $\mathrm{PSL}_2$, so $\tau_2$ is an elliptic point of order $2$. Solving $(W_3 T)\tau = \tau$, i.e., $-1/(3(\tau+1)) = \tau$, gives $3\tau^2 + 3\tau + 1 = 0$ and hence $\tau_6 = -\tfrac{1}{2} + \tfrac{i}{2\sqrt 3}$; a direct computation gives $(W_3 T)^6 = -27 \cdot I$ (scalar) with $(W_3 T)^k$ non-scalar for $1 \le k \le 5$, so $\tau_6$ is an elliptic point of order $6$. Moreover $\tau_2$ and $\tau_6$ are inequivalent under $\Gamma_0^+(3)$: the unordered pair $\{u(\tau),\,1-u(\tau)\}$ is $\Gamma_0^+(3)$-invariant---each entry is $\Gamma_0(3)$-invariant, and $W_3$ swaps the two entries by Lemma~\ref{lem:fricke-u}---and this pair equals $\{\tfrac12\}$ at $\tau_2$, since $u(\tau_2)=u(W_3\tau_2)=1-u(\tau_2)$, but equals $\{\infty\}$ at $\tau_6=\rho$, the pole of $u$ (Proposition~\ref{prop:schwarz-triangle}, Corollary~\ref{cor:u-v-hauptmodul}).

Gauss--Bonnet for a Fuchsian group of signature $(g; m_1, \ldots, m_r; s)$ acting on $\mathbb H$ reads
\[
\mathrm{area} \;=\; 2\pi\!\left[2g - 2 + s + \sum_{i=1}^r\!\left(1 - \tfrac{1}{m_i}\right)\right].
\]
The stabilizer of $\tau_2$ is a finite cyclic group containing the order-$2$ element $W_3$, so its order $m_1$ is even, giving $1-1/m_1\ge\tfrac12$; likewise the stabilizer of $\tau_6$ contains the order-$6$ element $W_3T$, so $6\mid m_2$ and $1-1/m_2\ge\tfrac56$. Substituting $s=1$ and letting $E\ge0$ denote the contribution of any further elliptic classes,
\[
\tfrac{2\pi}{3} \;=\; 2\pi\!\left[2g - 2 + 1 + \left(1-\tfrac1{m_1}\right) + \left(1-\tfrac1{m_2}\right) + E\right] \;\ge\; 2\pi\!\left[2g + \tfrac{1}{3} + E\right],
\]
with equality if and only if $m_1=2$ and $m_2=6$. Since $g\ge0$ and $E\ge0$, equality must hold throughout: $g=0$, $E=0$, $m_1=2$, $m_2=6$. In particular there are no further elliptic points beyond $\tau_2$ and $\tau_6$.
\end{proof}

\begin{lemma}[Dimension]
\label{lem:dim-weight12}
$\dim S_{12}(\Gamma_0^+(3))=2$.
\end{lemma}

\begin{proof}
For a Fuchsian group of finite covolume with signature $(g; m_1, \ldots, m_r; s)$ and $k$ even with $k \ge 4$, the standard Riemann--Roch dimension formula (see e.g.~\cite[\S2.6]{Shimura} or~\cite[\S2]{Miyake}) gives
\[
\dim S_k(\Gamma) \;=\; (k-1)(g-1) + \left(\tfrac{k}{2}-1\right)s + \sum_i\!\left\lfloor\tfrac{k}{2}\!\left(1 - \tfrac{1}{m_i}\right)\right\rfloor.
\]
This formula applies to $\Gamma_0^+(3)$ since all of its cusps are regular for even weight. With $(g; m_1, m_2; s) = (0; 2, 6; 1)$ and $k = 12$:
\[
\dim S_{12}(\Gamma_0^+(3)) \;=\; 11\cdot(-1) + 5\cdot 1 + \lfloor 3 \rfloor + \lfloor 5 \rfloor \;=\; -11 + 5 + 3 + 5 \;=\; 2. \qedhere
\]
\end{proof}

\begin{definition}[Two weight-$12$ cusp forms]
\label{def:delta-plus}
Let
\[
\Delta_3^+(\tau):=(\eta(\tau)\eta(3\tau))^{12},\qquad \Delta_1^+(\tau):=\Delta(\tau)+3^6\Delta(3\tau),
\]
where $\Delta(\tau)=\eta(\tau)^{24}$.
\end{definition}

\begin{lemma}[A basis of $S_{12}(\Gamma_0^+(3))$]
\label{lem:basis-S12}
Both $\Delta_1^+$ and $\Delta_3^+$ belong to $S_{12}(\Gamma_0^+(3))$, and $\{\Delta_1^+,\Delta_3^+\}$ is a basis. Moreover
\[
\Delta_1^+(\tau)=q+O(q^2),\qquad \Delta_3^+(\tau)=q^2(1+O(q)).
\]
\end{lemma}

\begin{proof}
\emph{$\Delta_3^+ \in S_{12}(\Gamma_0^+(3))$.} By Ligozat's criteria for eta-quotients (\cite[Thm.~1.64]{Ono}; the criteria go back to Ligozat~\cite[Prop.~3.2.1]{Ligozat}, and see also~\cite{Köhler}), the form $(\eta(\tau)\eta(3\tau))^{12}$ is a weight-$12$ cusp form on $\Gamma_0(3)$ with trivial character: $\sum d\,r_d = 12 + 36 = 48 \equiv 0 \pmod{24}$, $\sum (N/d)\,r_d = 36 + 12 = 48 \equiv 0 \pmod{24}$, and $\prod d^{r_d} = 3^{12} = (3^6)^2$ is a rational square. To verify $W_3$-invariance, the $\eta$-transformation $\eta(-1/\sigma) = \sqrt{-i\sigma}\,\eta(\sigma)$ gives
\[
\eta(W_3\tau)\,\eta(3\,W_3\tau) \;=\; \eta\!\left(-\tfrac{1}{3\tau}\right)\eta\!\left(-\tfrac{1}{\tau}\right) \;=\; \sqrt{-3i\tau}\,\eta(3\tau)\cdot\sqrt{-i\tau}\,\eta(\tau) \;=\; -i\sqrt 3\,\tau\,\eta(\tau)\eta(3\tau),
\]
with principal square roots (both radicands lie in the right half-plane for $\tau\in\mathbb H$), consistent with the weight-$2$ statement $\Delta_3(W_3\tau)=-3\tau^2\Delta_3(\tau)$ of Theorem~\ref{thm:modular-derivative}(ii). Raising to the $12$th power and applying the slash operator $(g|_{12} W_3)(\tau) = \det(W_3)^6\,(3\tau)^{-12}\,g(W_3\tau)$ with $\det(W_3) = 3$:
\[
\big(\Delta_3^+\big|_{12} W_3\big)(\tau) \;=\; 3^6\,(3\tau)^{-12}\,(-i\sqrt 3\,\tau)^{12}\,\Delta_3^+(\tau) \;=\; 3^6 \cdot 3^{-12}\,\tau^{-12} \cdot 3^6\,\tau^{12}\,\Delta_3^+(\tau) \;=\; \Delta_3^+(\tau).
\]
The $q$-expansion of $\Delta_3^+$ starts at $q^2$, so $\Delta_3^+ \in S_{12}(\Gamma_0^+(3))$. (The corresponding eigenvalue $v_\eta(W_3) = -1$ of the underlying weight-$2$ form $\eta(\tau)^2\eta(3\tau)^2$ was computed in Theorem~\ref{thm:modular-derivative}(ii).)

\emph{$\Delta_1^+ \in S_{12}(\Gamma_0^+(3))$.} Both $\Delta(\tau)$ and $\Delta(3\tau)$ lie in $S_{12}(\Gamma_0(3))$ by standard eta-quotient theory. Under $W_3$, using $\Delta(-1/\sigma) = \sigma^{12}\,\Delta(\sigma)$,
\[
\Delta\big|_{12} W_3 \;=\; 3^6\,\Delta(3\tau), \qquad \Delta(3\cdot)\big|_{12} W_3 \;=\; 3^{-6}\,\Delta(\tau).
\]
Therefore
\[
\big(\Delta + 3^6\,\Delta(3\cdot)\big)\big|_{12} W_3 \;=\; 3^6\,\Delta(3\cdot) + 3^6\cdot 3^{-6}\,\Delta \;=\; \Delta + 3^6\,\Delta(3\cdot) \;=\; \Delta_1^+.
\]
Combined with $T$-invariance (both $\Delta$ and $\Delta(3\cdot)$ have period $1$), this gives $\Delta_1^+ \in S_{12}(\Gamma_0^+(3))$. Since $\Delta = q + O(q^2)$ and $\Delta(3\tau)$ starts at $q^3$, $\Delta_1^+$ has $q$-expansion $q + O(q^2)$.

\emph{Linear independence and basis.} The leading $q$-expansions $\Delta_1^+ = q + O(q^2)$ and $\Delta_3^+ = q^2 + O(q^3)$ are linearly independent, and $S_{12}(\Gamma_0^+(3))$ is two-dimensional by Lemma~\ref{lem:dim-weight12}.
\end{proof}

\begin{lemma}[Valence formula for $\Gamma_0^+(3)$]
\label{lem:valence}
For a nonzero $G \in M_{12}(\Gamma_0^+(3))$,
\[
\mathrm{ord}_\infty(G) + \tfrac{1}{2}\,\mathrm{ord}_{\tau_2}(G) + \tfrac{1}{6}\,\mathrm{ord}_{\tau_6}(G) + \!\!\sum_{\substack{z \in \Gamma_0^+(3)\backslash\mathbb H \\ z \ne \tau_2, \tau_6}}\!\!\mathrm{ord}_z(G) \;=\; \frac{12}{4\pi}\cdot \mathrm{area}(\Gamma_0^+(3)\backslash\mathbb H) \;=\; \frac{12}{4\pi}\cdot\frac{2\pi}{3} \;=\; 2,
\]
where $\tau_2 = i/\sqrt 3$ and $\tau_6 = -\tfrac{1}{2} + \tfrac{i}{2\sqrt 3}$ are the elliptic points of orders $2$ and $6$ from Lemma~\ref{lem:signature-Fricke}, and $\mathrm{ord}_z(G)$ denotes the order of vanishing of $G$ as a holomorphic function at $z$.
\end{lemma}

\begin{proof}
This is the valence formula~\cite[\S2.3--2.4]{Shimura} for weight $k = 12$, with $\mathrm{area}(\Gamma_0^+(3)\backslash\mathbb H) = 2\pi/3$ as computed in the proof of Lemma~\ref{lem:signature-Fricke}.
\end{proof}

\subsection{Proof of Theorem~\ref{thm:uniqueness}}

Let
\[
T=\begin{pmatrix}1&1\\0&1\end{pmatrix},\qquad W_3=\begin{pmatrix}0&-1\\3&0\end{pmatrix},\qquad U:=W_3 T,
\]
so that $T\tau=\tau+1$, $W_3\tau=-1/(3\tau)$, $U\tau=-1/(3(\tau+1))$. Recall from \S\ref{sec:proof-main} the weight-$k$ slash operator on $\mathrm{GL}_2^+(\mathbb R)$:
\[
(g|_k M)(\tau)=\det(M)^{k/2}(c\tau+d)^{-k}\,g(M\tau),\quad M=\begin{pmatrix}a&b\\c&d\end{pmatrix}.
\]

\paragraph{Step 1 ($T$-eigenrelation).} Each term $q^{n+\alpha}=e^{2\pi i(n+\alpha)\tau}$ is multiplied by $e^{2\pi i(n+\alpha)}=e^{2\pi i\alpha}$ under $\tau\mapsto\tau+1$, so the $q$-expansion gives
\[
F(\tau+1)=A\,F(\tau),\qquad A:=e^{2\pi i\alpha},
\]
for $\Im\tau$ large enough that the expansion converges; since both sides are holomorphic on $\mathbb H$, the identity theorem extends the relation to all of $\mathbb H$. Differentiating, $f(\tau+1)=A f(\tau)$, i.e.\ $f|_2 T=A f$.

\paragraph{Step 2 (Fricke relation for $f$).} Differentiating $F(W_3\tau)=1-F(\tau)$ gives $f(W_3\tau)/(3\tau^2)=-f(\tau)$. Since $\det(W_3)=3$ and $(3\tau)^{-2}=1/(9\tau^2)$, we compute
\[
(f|_2 W_3)(\tau)=3\cdot(3\tau)^{-2}\,f(W_3\tau)=\frac{1}{3\tau^2}f(W_3\tau)=-f(\tau),
\]
so $f|_2 W_3=-f$.

\paragraph{Step 3 (finite-order relation and $\alpha \in \{\tfrac16, \tfrac13, \tfrac23, \tfrac56\}$).} From $F(T\tau) = A\,F(\tau)$ and $F(W_3\tau) = 1 - F(\tau)$ we obtain
\[
F(U\tau) \;=\; 1 - A\,F(\tau).
\]
Setting $G(x) := 1 - Ax$, we have $F(U^n\tau) = G^{\circ n}(F(\tau))$. A direct calculation gives $U^6 = -27\,I$ in $\mathrm{GL}_2^+(\mathbb R)$, so $U^6$ acts trivially on $\mathbb H$ as a M\"obius transformation. Therefore $G^{\circ 6}$ fixes every point of $F(\mathbb H)$; since $a_0\ne0$ forces $F$ to be nonconstant, $F(\mathbb H)$ is open, and an affine map of $\mathbb C$ fixing a nonempty open set is the identity. Hence $G^{\circ 6} = \mathrm{id}$. Expanding,
\[
G^{\circ 3}(x) \;=\; 1 - A + A^2 - A^3 x, \qquad G^{\circ 6}(x) \;=\; (1 - A + A^2)(1 - A^3) + A^6\,x.
\]
Thus $G^{\circ 6} = \mathrm{id}$ forces $A^6 = 1$ \emph{and} $(1 - A + A^2)(1 - A^3) = 0$. Since $A=e^{2\pi i\alpha}$, the condition $A^6=1$ says $6\alpha\in\mathbb Z$, and constrains $\alpha$ only modulo $1$; the auxiliary equation eliminates exactly the class $\alpha \equiv \tfrac12 \pmod 1$ (where $A = -1$ and $(1 + 1 + 1)(1 + 1) = 6 \ne 0$), while the classes $\alpha\equiv 0,\tfrac16,\tfrac13,\tfrac23,\tfrac56 \pmod 1$ all satisfy it. Hence, with $\alpha>0$,
\[
6\alpha \;\in\; \mathbb Z_{>0}, \qquad \alpha \;\not\equiv\; \tfrac12 \pmod 1,
\]
that is, $\alpha \in \{\tfrac16, \tfrac13, \tfrac23, \tfrac56, 1, \tfrac76, \dots\}$, the positive multiples of $\tfrac16$ omitting $\tfrac12+\mathbb Z_{\ge0}$.
Differentiating the $q$-expansion of $F$ in $\tau$, $f(\tau) = 2\pi i\sum_{n\ge 0} a_n(n + \alpha)\,q^{n + \alpha} = c_0\,q^\alpha(1 + O(q))$ with $c_0 = 2\pi i \alpha\, a_0 \ne 0$. Therefore $f^6 = c_0^6\,q^{6\alpha}(1 + O(q))$. Because $A^6 = 1$ and $(-1)^6 = 1$, Steps 1 and 2 give $f^6|_{12} T = f^6$ and $f^6|_{12} W_3 = f^6$. Now $T$ and $W_3$ generate $\Gamma_0^+(3)$ modulo the center: indeed $W_3 T W_3^{-1} = \left(\begin{smallmatrix}1&0\\-3&1\end{smallmatrix}\right) = L^{-1}$, and $\Gamma_0(3)$ is generated by $T$ and $L$ modulo $\pm I$ (Theorem~\ref{thm:monodromy-Gamma03}), so $\langle T, W_3\rangle \supseteq \langle \Gamma_0(3), W_3\rangle = \Gamma_0^+(3)$ in $\mathrm{PSL}_2(\mathbb R)$; the center acts trivially in even weight. We obtain
\[
f^6 \;\in\; M_{12}(\Gamma_0^+(3)).
\]
Since $c_0 \ne 0$ and $6\alpha > 0$, the form $f^6$ vanishes at the cusp; hence $f^6 \in S_{12}(\Gamma_0^+(3))$.

\paragraph{Step 4 (nonvanishing forces $\alpha = \tfrac13$ and $f^6 = c_0^6\,\Delta_3^+$).}
By hypothesis~(iii), $f$ is nowhere vanishing on $\mathbb H$, so $f^6$ is nowhere vanishing on $\mathbb H$ as well. In the valence formula (Lemma~\ref{lem:valence}) applied to $G = f^6$, every term coming from a point of $\mathbb H$ therefore vanishes---including the elliptic contributions $\tfrac12\,\mathrm{ord}_{\tau_2}(f^6)$ and $\tfrac16\,\mathrm{ord}_{\tau_6}(f^6)$---and the formula collapses to
\[
\mathrm{ord}_\infty(f^6) \;=\; 2.
\]
Since $\mathrm{ord}_\infty(f^6) = 6\alpha$ by the leading $q$-term computed in Step 3, this gives $\alpha = \tfrac13$. (In particular every other admissible value is excluded at once: any admissible $\alpha > \tfrac13$ would give $\mathrm{ord}_\infty(f^6) = 6\alpha \ge 4 > 2$, exceeding the valence total, while $\alpha = \tfrac16$ would give $\mathrm{ord}_\infty(f^6) = 1$, forcing an interior zero of $f^6$ and contradicting nonvanishing.)

Now write $f^6 = a\,\Delta_1^+ + b\,\Delta_3^+$ in the basis of Lemma~\ref{lem:basis-S12}. Since $\Delta_1^+ = q + O(q^2)$ and $\Delta_3^+ = q^2 + O(q^3)$, the coefficient of $q^1$ in $f^6$ equals $a$; but $\mathrm{ord}_\infty(f^6) = 2$ means this coefficient vanishes, so $a = 0$. Hence
\[
f(\tau)^6 \;=\; b\,\Delta_3^+(\tau) \;=\; b\,(\eta(\tau)\eta(3\tau))^{12}, \qquad b = c_0^6 \ne 0.
\]
Because $\eta(\tau)$ and $\eta(3\tau)$ are nowhere vanishing on $\mathbb H$, the function $h := f/(\eta^2\eta(3\cdot)^2)$ is holomorphic and nowhere vanishing on $\mathbb H$ with $h^6 = b$ constant. A continuous function on the connected set $\mathbb H$ valued in the finite set of sixth roots of $b$ is constant, so
\[
f(\tau) \;=\; C\,\eta(\tau)^2\,\eta(3\tau)^2, \qquad C^6 = b.
\]

\paragraph{Recovering $F$.} Let $\widetilde\Pi_h^{\mathrm{par}}$ denote the holomorphic extension to $\mathbb H$ of $\tau\mapsto\Pi_h^{\mathrm{par}}(-i\sqrt 3\,\tau)$ provided by Theorem~\ref{thm:modular-derivative}(i). Since $\alpha = 1/3 > 0$, integrating $f = C\,\eta^2\eta(3\cdot)^2$ from the cusp gives
\[
F(\tau) \;=\; C\int_{i\infty}^\tau \eta(z)^2\,\eta(3z)^2\,dz, \qquad F(\tau) \to 0 \ \text{as}\ \Im\tau \to \infty.
\]
By Theorem~\ref{thm:modular-derivative}(i), $\widetilde\Pi_h^{\mathrm{par}}$ is the same integral with constant $c_\eta$ in place of $C$, and it too vanishes at the cusp; hence $F = \kappa\,\widetilde\Pi_h^{\mathrm{par}}$ with $\kappa := C/c_\eta$ (two antiderivatives of proportional forms that both vanish at $i\infty$).

It remains to fix the scalar $\kappa$, and this is where the functional equation re-enters. Both functions satisfy~(ii): $F$ does by hypothesis, and $\widetilde\Pi_h^{\mathrm{par}}$ does because its integrand has Fricke eigenvalue $\eta^2\eta(3\cdot)^2\big|_2 W_3 = -\,\eta^2\eta(3\cdot)^2$ (Theorem~\ref{thm:modular-derivative}(ii)), so that $\widetilde\Pi_h^{\mathrm{par}}(W_3\tau) + \widetilde\Pi_h^{\mathrm{par}}(\tau)$ is constant on $\mathbb H$; evaluating as $\Im\tau\to\infty$ and using the boundary values $\Pi_h^{\mathrm{par}}(r)\to 0$ as $r\to\infty$ and $\Pi_h^{\mathrm{par}}(r)\to 1$ as $r\to 0^+$ (both from the integral formula of Theorem~\ref{thm:cardy-parallelogram}(i) together with the endpoint limits of Lemma~\ref{lem:endpoint-r}) shows the constant is $1$. Substituting $F = \kappa\,\widetilde\Pi_h^{\mathrm{par}}$ into~(ii) now gives
\[
1 \;=\; F(W_3\tau) + F(\tau) \;=\; \kappa\bigl[\widetilde\Pi_h^{\mathrm{par}}(W_3\tau) + \widetilde\Pi_h^{\mathrm{par}}(\tau)\bigr] \;=\; \kappa,
\]
so $\kappa = 1$ and $F = \widetilde\Pi_h^{\mathrm{par}}$, completing the proof.\qed

\begin{remark}[The necessity of hypothesis (iii): a spurious solution]
\label{rem:ghost}
Hypothesis~(iii) cannot be omitted. Without it, Step~4 only yields $\mathrm{ord}_\infty(f^6) = 6\alpha \le 2$, which admits $\alpha = \tfrac16$ in addition to $\alpha = \tfrac13$. For $\alpha = \tfrac16$ one has $\mathrm{ord}_\infty(f^6) = 1$, and writing $f^6 = a\,\Delta_1^+ + b\,\Delta_3^+$ the cusp order fixes only $a = c_0^6 \ne 0$, leaving the coefficient $b$ \emph{undetermined}. Crucially, the order of any $G \in M_{12}(\Gamma_0^+(3))$ at the order-$6$ elliptic point $\tau_6$ is divisible by $6$ (near $\tau_6$, the automorphy law under the order-$6$ stabilizer generated by $U$, whose derivative at the fixed point is $U'(\tau_6) = e^{-i\pi/3}$, forces the Taylor expansion of $G$ at $\tau_6$ to contain only the powers $(\tau-\tau_6)^{6j}$), so choosing $b = -a\,\Delta_1^+(\tau_6)/\Delta_3^+(\tau_6)$ to make $f^6$ vanish at $\tau_6$ forces a zero of order exactly $6$ there; by Lemma~\ref{lem:valence} this accounts for the entire valence, and $f := (f^6)^{1/6}$ is then a genuine holomorphic weight-$2$ form on $\mathbb H$ with a simple zero at $\tau_6$. The ratio can be evaluated exactly. With $\gamma=\left(\begin{smallmatrix}0&-1\\1&1\end{smallmatrix}\right)\in\mathrm{SL}_2(\mathbb Z)$ and $\mu:=e^{i\pi/3}$ one has $\gamma\mu=-1/(\mu+1)=\tau_6$, so $\Delta(\tau_6)=(\mu+1)^{12}\Delta(\mu)=3^6\,\Delta(\mu)$, using $\mu+1=\sqrt3\,e^{i\pi/6}$; and $\Delta(3\tau_6)=\Delta(3\tau_6+2)=\Delta(\mu)$. Hence $\Delta_1^+(\tau_6)=2\cdot3^6\,\Delta(\mu)$ and $\Delta_3^+(\tau_6)^2=\Delta(\tau_6)\,\Delta(3\tau_6)=\bigl(27\,\Delta(\mu)\bigr)^2$. On the line $\Re\tau=-\tfrac12$ the nome $q=e^{2\pi i\tau}$ is real and negative, so $\Delta_3^+=q^2\prod_{n\ge1}(1-q^n)^{12}(1-q^{3n})^{12}>0$ there, while $q(\mu)=-e^{-\pi\sqrt3}<0$ gives $\Delta(\mu)<0$; therefore $\Delta_3^+(\tau_6)=-27\,\Delta(\mu)$, and
\[
\frac{b}{a}\;=\;-\,\frac{\Delta_1^+(\tau_6)}{\Delta_3^+(\tau_6)}\;=\;-\,\frac{2\cdot 3^6\,\Delta(\mu)}{-27\,\Delta(\mu)}\;=\;54 .
\]
One checks further that $f|_2 T = e^{i\pi/3}f$ and $f|_2 W_3 = -f$, and that the resulting $F = \int_{i\infty}^\tau f$ satisfies $F(W_3\tau) + F(\tau) = K$ for a nonzero constant $K = 2F(\tau_2)$; rescaling by $1/K$ produces a function satisfying all of~(i)--(ii) with $\alpha = \tfrac16 \ne \tfrac13$. Thus the Fricke symmetry and $q$-block ansatz alone do not determine the Cardy--Smirnov function at level $3$. The spurious solution is excluded precisely by~(iii): its derivative vanishes at $\tau_6$, whereas the conformal-geometric origin of $\Pi_h^{\mathrm{par}}$ forces $d\Pi_h^{\mathrm{par}}/d\tau$ to be nowhere zero. Geometrically, $\tau_6$ coincides with the vertex $\rho=\tau(\infty)$ of the Schwarz triangle $\mathcal T$ (Proposition~\ref{prop:schwarz-triangle}); it is the pole of the hauptmodul, the $\pi/3$ cone point, where the branch factors of the conformal map cancel the ramification of $v$ to leave a nonvanishing derivative; the spurious solution places a critical point exactly where the genuine one cannot have one. This phenomenon has no analogue in the rectangular case, where $\dim S_{12}(\mathrm{SL}_2(\mathbb Z)) = 1$.
\end{remark}

\section{Discussion and further work}
\label{sec:discussion}

\subsection{Summary}

We have proved three $\pi/3$-parallelogram results. First, the Cardy--Smirnov function on the $\pi/3$ parallelogram admits a closed conformal-map formula as an incomplete beta integral (Theorem~\ref{thm:cardy-parallelogram}), in direct analogy with Cardy's original 1992 formula on the rectangle. Second, this function admits a closed formula as an integral of the modular form $\eta(\tau)^2\eta(3\tau)^2$ (Theorem~\ref{thm:main-almost-modular}), which is a weight-$2$ cusp form on $\Gamma_0(3)$ with respect to a non-trivial cubic multiplier system. Third, the Cardy--Smirnov function is characterized by the Fricke symmetry $F(-1/(3\tau)) = 1 - F(\tau)$ together with a $q$-block ansatz and the nondegeneracy condition that $F'$ be nowhere vanishing (Theorem~\ref{thm:uniqueness}).

The comparison with~\cite{KZ} is instructive. In the rectangular case, adjoining the Fricke-type involution produces the full modular group $\mathrm{SL}_2(\mathbb Z)$, and the corresponding weight-$12$ cusp form space is one-dimensional (spanned by $\Delta$). There the Fricke symmetry and $q$-block ansatz alone determine the function. In the level-$3$ setting, even after adjoining the Fricke involution one remains on the smaller Fricke extension $\Gamma_0^+(3)$, whose weight-$12$ cusp-form space is two-dimensional. As a result the symmetry hypotheses fall one condition short, and there is a genuine second solution (Remark~\ref{rem:ghost}); one must impose the geometric nondegeneracy condition that $F'$ not vanish---automatic for a conformal crossing probability---to single out the Cardy--Smirnov function. This is the essential new feature of the parallelogram relative to the rectangle.

\subsection{Crossing probabilities for larger \texorpdfstring{$n$}{n}}

For other values of $n$, one may ask whether the $\pi/n$ parallelogram admits an analogous modular reformulation of its Cardy--Smirnov function. The construction of \S\ref{sec:SC}--\S\ref{sec:proof-main} separates into two ingredients:

\begin{itemize}
\item \emph{Existence of a modular formula:} the projective monodromy of the Schwarz--Christoffel hypergeometric equation must be arithmetic, i.e.\ commensurable with $\mathrm{SL}_2(\mathbb Z)$. For the $\pi/n$ parallelogram the Schwarz--Christoffel integrand is $w^{1/n-1}(w-u)^{-1/n}(w-1)^{1/n-1}$, and the computation of Theorem~\ref{thm:r-ratio} goes through unchanged to give
\[
r(u)=\frac{F_n(u)}{F_n(1-u)},\qquad F_n:={}_2F_1\!\left(\tfrac1n,1-\tfrac1n;1;\,\cdot\,\right),
\]
with $F_n$ the hypergeometric function of Ramanujan's signature-$n$ theory. The exponent differences of the associated Gauss equation are $\lambda_0=\lambda_1=0$ and $\lambda_\infty=1-\tfrac2n$, so the Schwarz triangle has angles $0,0,\pi(1-\tfrac2n)$. The reflections in its sides generate a discrete group exactly when $\pi(1-\tfrac2n)=\pi/q$ with $q\in\mathbb Z_{\ge2}\cup\{\infty\}$, that is, when $q=n/(n-2)$ is an integer; since $(n-2)\mid n$ forces $(n-2)\mid 2$, this happens only for $n=2,3,4$, with $q=\infty,3,2$ and triangle groups $\Gamma(2)$, $\Gamma_0(3)$, $\Gamma_0(2)$ respectively. For those three values the argument of~\S\ref{sec:monodromy} applies verbatim. In particular $n=2$ recovers the rectangle of Kleban and Zagier~\cite{KZ}: there $v=\theta_2^4/\theta_3^4$ and $F_2=\theta_3^2$, so that by Jacobi's identity $\theta_2\theta_3\theta_4=2\eta^3$,
\[
v^{1/3}(1-v)^{1/3}F_2(v)^2=(\theta_2\theta_3\theta_4)^{4/3}=2^{4/3}\,\eta(\tau)^4,
\]
in place of the identity $bc=3\,\eta(\tau)^2\eta(3\tau)^2$ used in Theorem~\ref{thm:modular-derivative}. The value $n=6$ is different in kind: $\lambda_\infty=\tfrac23$ is not the reciprocal of an integer, the reflected triangles overlap, and the inverse Schwarz map is not single-valued on $\mathbb H$, so Corollary~\ref{cor:u-v-hauptmodul} has no direct analogue; we do not attempt to settle it here (compare Takeuchi's classification of arithmetic triangle groups~\cite{Takeuchi}).

\item \emph{Uniqueness:} the proof of Theorem~\ref{thm:uniqueness} reduces, via the nondegeneracy condition, to the small dimension $\dim S_{12}(\Gamma_0^+(3)) = 2$. For larger levels the relevant cusp-form spaces grow, and even with a nondegeneracy condition the Fricke symmetry and $q$-block ansatz may no longer determine the Cardy--Smirnov function uniquely; identifying sufficient additional constraints is open.
\end{itemize}

This suggests two natural questions:
\begin{enumerate}
\item For $n \in \{4, 6\}$, what is the closed modular formula for the Cardy--Smirnov function on the $\pi/n$ parallelogram, and what eta quotient plays the role of $\eta(\tau)^2\eta(3\tau)^2$?
\item When the weight-$k$ cusp form space at the relevant level has dimension $\ge 2$, what additional hypotheses---beyond the Fricke symmetry, the $q$-block ansatz, and the nonvanishing of $F'$ used here---are needed to characterize the Cardy--Smirnov function uniquely? Already at level $3$ the nonvanishing condition is essential (Remark~\ref{rem:ghost}).
\end{enumerate}

\appendix

\section{The Schwarz--Christoffel derivative for \texorpdfstring{$P_r$}{P\_r}}
\label{app:SC}

We recall how \eqref{eq:fu-def} arises. For a Schwarz--Christoffel map from $\mathbb H$ onto a polygon with interior angles $\alpha_1\pi,\dots,\alpha_n\pi$ at prevertices $x_1,\dots,x_n\in\mathbb R\cup\{\infty\}$, the derivative has the form
\[
f'(z)=C\prod_{j=1}^n (z-x_j)^{\alpha_j-1},
\]
with the factor at $\infty$ absorbed into $C$~\cite[Ch.~8]{Stein}. For $P_r$ the interior angles are $(\tfrac{\pi}{3},\tfrac{2\pi}{3},\tfrac{\pi}{3},\tfrac{2\pi}{3})$, so the exponents are $-\tfrac23,-\tfrac13,-\tfrac23,-\tfrac13$. Choosing prevertices $(0,u,1,\infty)$ and absorbing the factor at $\infty$ yields
\[
f_u'(z)=C_{\mathrm{SC}}(u)\,z^{-2/3}(z-u)^{-1/3}(z-1)^{-2/3},
\]
and integrating from $z_0=0$ with $A_0=f_u(0)=0$ gives \eqref{eq:fu-def}.

\section{Endpoint estimates for the period ratio}
\label{app:endpoint}

We prove Lemma~\ref{lem:endpoint-r}. Recall
\[
N(u)=\int_0^1 t^{-2/3}(1-t)^{-1/3}(1-ut)^{-2/3}\,dt,\qquad D(u)=\int_0^1(u+(1-u)t)^{-2/3}t^{-1/3}(1-t)^{-2/3}\,dt,
\]
obtained from the integrals in Definition~\ref{def:AB} by the substitutions $w=ut$ and $w=u+(1-u)t$.

\paragraph{Behavior as $u\to 1^-$.} Restricting the integral defining $N(u)$ to $t\in[1/2,1]$ and writing $s=1-t$,
\[
N(u)\ge\int_0^{1/2}s^{-1/3}(1-u+us)^{-2/3}\,ds\ge 2^{-2/3}\int_{1-u}^{1/2}s^{-1}\,ds=2^{-2/3}\log\frac{1}{2(1-u)}\to+\infty,
\]
using $1-u+us\le (1-u)+s$. For $D(u)$, dominated convergence (with bound $2^{2/3}$ on $u\in[1/2,1)$) gives $D(u)\to\int_0^1 t^{-1/3}(1-t)^{-2/3}\,dt=B(\tfrac23,\tfrac13)$. Hence $r(u)=N(u)/D(u)\to+\infty$.

Conversely, fix $\delta\in(0,1)$ and assume $u\le 1-\delta$. Then $1-ut\ge\delta$, so $N(u)\le\delta^{-2/3}B(\tfrac13,\tfrac23)$. On the other hand, since $[1-\delta/2,1]\subset[u,1]$, reverting to the original variable gives $D(u)\ge\int_{1-\delta/2}^1(1-w)^{-2/3}\,dw=3(\delta/2)^{1/3}$. Therefore $r(u)\le C_\delta:=\delta^{-2/3}B(\tfrac13,\tfrac23)/(3(\delta/2)^{1/3})<\infty$, so $r(u)\to+\infty$ forces $u\to 1^-$.

\paragraph{Behavior as $u\to 0^+$.} By dominated convergence, $N(u)\to\int_0^1 t^{-2/3}(1-t)^{-1/3}\,dt=B(\tfrac13,\tfrac23)$ as $u\to 0^+$. For $D(u)$, if $0<u<1/4$ then on $[2u,1/2]$ we have $w-u\le w$ and $1-w\le 1$, so reverting to the original variable,
\[
D(u)\ge\int_{2u}^{1/2}w^{-2/3}w^{-1/3}\,dw=\log\frac{1}{4u}\to+\infty.
\]
Hence $r(u)\to 0^+$. For the converse, for $u\ge\delta$ one has $N(u)\ge(\delta/2)^{1/3}$ (restrict to $w\in(0,\delta/2]$, where $(u-w)^{-1/3}\ge 1$) and $D(u)\le\delta^{-2/3}B(\tfrac23,\tfrac13)$ (since $w\ge\delta$ on $[u,1]$). Thus $r(u)\ge c_\delta>0$ uniformly in $u\ge\delta$, so $r(u)\to 0^+$ forces $u\to 0^+$.

\paragraph{Monotonicity and analyticity.} In the first displayed form of $N$ above, the factor $(1-ut)^{-2/3}$ is, for each fixed $t\in(0,1)$, strictly increasing in $u\in(0,1)$; hence $N$ is strictly increasing. Likewise, in the displayed form of $D$, the factor $(u+(1-u)t)^{-2/3}$ is strictly decreasing in $u$ for each fixed $t\in(0,1)$, because $\frac{\partial}{\partial u}\bigl(u+(1-u)t\bigr)=1-t>0$; hence $D$ is strictly decreasing. Therefore $r=N/D$ is strictly increasing on $(0,1)$. Both $N$ and $D$ are restrictions of the hypergeometric functions $B(\tfrac13,\tfrac23)\,G(u)$ and $B(\tfrac13,\tfrac23)\,G(1-u)$ (proof of Theorem~\ref{thm:r-ratio}), hence real-analytic on $(0,1)$, and $D>0$ there, so $r$ is real-analytic. Together with the endpoint limits established above and the intermediate value theorem, $r$ is a real-analytic strictly increasing bijection $(0,1)\to(0,\infty)$, which is the statement of Lemma~\ref{lem:endpoint-r}.

\section{Acknowledgements}
The author would like to thank his advisor Dan Romik, who suggested the problem addressed in this paper, the writing of which would not have been possible without his insight and guidance throughout. The author would also like to extend his thanks to Elena Fuchs, Anne Schilling, Andrew Waldron, and Jaroslav Trnka for their questions and discussion during the author's qualifying exam.

\end{document}